\documentclass[12pt]{amsart}
\usepackage{a4wide,enumerate,xcolor,graphicx}
\usepackage{amsmath}
\usepackage{float}
\allowdisplaybreaks

\let\pa\partial
\let\na\nabla
\let\eps\varepsilon
\newcommand{\N}{{\mathbb N}}
\newcommand{\R}{{\mathbb R}}
\newcommand{\diver}{\operatorname{div}}

\numberwithin{equation}{section}
\newtheorem{theorem}{Theorem}[section]
\newtheorem{definition}[theorem]{Definition}
\newtheorem{lemma}[theorem]{Lemma}

\newtheorem{proposition}[theorem]{Proposition}

\begin{document}

\title[Cahn--Hilliard cross-diffusion systems]{Existence of global weak solutions \\ to a Cahn--Hilliard cross-diffusion system \\ in lymphangiogenesis}

\author[A. J\"ungel]{Ansgar J\"ungel}
\address{Institute of Analysis and Scientific Computing, Technische Universit\"at, Wiedner Hauptstra\ss e 8--10, 1040 Wien, Austria}
\email{juengel@tuwien.ac.at}

\author[Y. Li]{Yue Li}
\address{Institute of Analysis and Scientific Computing, Technische Universit\"at, Wiedner Hauptstra\ss e 8-10, 1040 Wien, Austria}
\email{yue.li@asc.tuwien.ac.at}

\date{\today}

\thanks{The authors acknowledge partial support from
the Austrian Science Fund (FWF), grants P33010 and F65. This work has received funding from the European Research Council (ERC) under the European Union's Horizon 2020 research and innovation programme, ERC Advanced Grant no.~101018153.}

\begin{abstract}
The global-in-time existence of weak solutions to a degenerate Cahn--Hilliard cross-diffusion system with singular potential in a bounded domain with no-flux boundary conditions is proved. The model consists of two coupled parabolic fourth-order partial differential equations and describes the evolution of the fiber phase volume fraction and the solute concentration, modeling the pre-patterning of lymphatic vessel morphology. The fiber phase fraction satisfies the segregation property if this holds initially. The existence proof is based on a three-level approximation scheme and a priori estimates coming from the energy and entropy inequalities. While the free energy is nonincreasing in time, the entropy is only bounded because of the cross-diffusion coupling.
\end{abstract}



\maketitle


\section{Introduction}

Lymphangiogenesis describes the formation of lymphatic vessels from pre-existing ones similar to angiogenesis. In this paper, we analyze a two-phase diffusion system modeling the pre-patterning of lymphatic vessel morphology in collagen gels. The model, first suggested in \cite{RoFo08}, describes the interaction of the collagen gel with a solute, such as protons and nutrients. The equations have been modified and given a thermodynamically consistent form in \cite{JuWa23}. They describe the evolution of the volume fractions of the fiber phase $\phi(x,t)$ and fluid phase $1-\phi(x,t)$ as well as the concentration $c(x,t)$ of the solute. The unknowns are solutions to the cross-diffusion equations
\begin{align}
  \pa_t\phi &= \diver\big(M(\phi)(\na\mu - c\na\pa_cf(\phi,c)) \big),
  \label{m1} \\
  \pa_tc &= -\diver\big(cM(\phi)(\na\mu - c\na\pa_c f(\phi,c)) \big)
  + \diver\big(ce^{-\phi}\na\pa_c f(\phi,c) \big), \label{m2} \\
  \mu &= -\Delta \phi+\partial_\phi f(\phi,c)\quad\mbox{in }\Omega,\ t>0, \label{m3}
\end{align}
where $\Omega\subset\R^d$ ($d\ge 1$) is a bounded domain and $\pa_c=\pa/\pa c$, $\pa_\phi=\pa/\pa\phi$ are partial derivatives. The (degenerate) mobility is given as in \cite{RoFo08} by
\begin{equation}\label{1.M}
  M(\phi) = \phi^2(1-\phi)^2,
\end{equation}
and the energy density reads as
\begin{equation}\label{1.f}
  f(\phi,c) = \phi\log\phi + (1-\phi)\log(1-\phi) + \phi(1-\phi)
  + \frac{c^2}{2} + c(1-\phi).
\end{equation}
The first three terms represent the (nonconvex) Flory--Huggins energy \cite{Flo42,Hug41}, and the last two terms are the nutrient energy \cite[(2.63)]{GKT22}. The potential $\pa_\phi f(\phi,c)$ in \eqref{m2} contains the term $\log(\phi/(1-\phi))$ which is singular at $\phi=0$ and $\phi=1$. The coefficient $ce^{-\phi}$ can be interpreted as the effective diffusion of solute in the fluid phase \cite[(4)]{RoFo08}. Observe that the corresponding term is different from the diffusion in \cite{RoFo08} because of thermodynamic considerations; see \cite[Section 2]{JuWa23} for details on the derivation of the model.

Equations \eqref{m1}--\eqref{m3} are supplemented by initial and no-flux boundary conditions,
\begin{align}
  \phi(0) = \phi_0, \quad c(0) = c_0 &\quad\mbox{in }\Omega, \label{ic} \\
  \na\phi\cdot\nu = c\na c\cdot\nu = cM(\phi)\na\mu\cdot\nu = 0
  &\quad\mbox{on }\pa\Omega,\ t>0. \label{bc}
\end{align}
Model \eqref{m1}--\eqref{m3} is a fourth-order cross-diffusion system with the following features. If the chemical potential $\mu$ is constant, the diffusion matrix associated to the variables $(\phi,c)$ has a vanishing eigenvalue. This indicates that it is more convenient to work with the thermodynamic variables $(\mu,\pa_c f)$ (see below). If the nutrient energy is constant, we obtain the Cahn--Hilliard equation for phase separation with a nonconvex energy. Our aim is to prove the global existence of weak solutions to \eqref{m1}--\eqref{bc}.

\subsection{State of the art}

The study of two-phase models is stimulated by various applications modeling, for instance, tumor growth \cite{EGN21}, biofilm growth \cite{WaZh12}, or formation of lymphatic vessels \cite{RoFo08}. The mathematical modeling of lymphangiogenesis is rather recent. A discrete compartment model of the lymphatic system was already presented in the 1970s \cite{Red74}. More recently, a differential equations model was presented in \cite{BPS15}. The work \cite{FrLo05} analyzed a diffusion system with haptotaxis and chemotaxis terms for tumor lymphangiogenesis. The collagen pre-pattering caused by interstitial fluid flow is modeled in \cite{RoFo08} by Cahn--Hilliard-type equations. It was found in \cite{RoSw11} that a hexagonal lymphatic capillary network is optimal in terms of fluid drainage, confirmed by experiments in mouse tails and human skin. This hexagonal structure was also found in numerical simulations in two space dimensions \cite{JuWa23}.

The Cahn--Hilliard model was introduced in \cite{CaHi58} to study phase separations in binary alloys. The first existence analysis of Cahn--Hilliard equations was given in \cite{ElSo86} in one space dimension and in \cite{ElGa96} in several space dimensions. The existence and uniqueness of solutions to the Cahn--Hilliard system strictly depend on the properties of the mobility $M(\phi)$ (degenerate or nondegenerate) and the potential $\pa_\phi f(\phi,c)$ (singular or regular). Since the Cahn--Hilliard equations do not admit a comparison principle, lower and upper bounds for the phase variable cannot generally be obtained. A sufficient condition for the property $0\le\phi\le 1$ (if satisfied initially) is a degenerate mobility $M(\phi)$ \cite{Yin92} or a singular potential \cite{BlEl91}. We refer to the book \cite{Mir19} for more details and references.

Diffusion systems with Cahn--Hilliard terms were analyzed more recently, in particular in the context of biological membranes \cite{GKRR16} and tumor growth modeling \cite{RSS23}. In these models, the cross diffusion is of Keller--Segel type and thus, the diffusion matrix is triangular, which simplifies the analysis. In \cite{HJT23}, Maxwell--Stefan models for fluid mixtures with full diffusion matrix and Cahn--Hilliard-type chemical potentials were investigated. Finally, a Cahn--Hilliard cross-diffusion model arising in thin-film solar cell fabrication was analyzed in \cite{CEMP23,EMP21}. The techniques of these papers, however, cannot be employed for our model. In particular, the existence analysis of model \eqref{m1}--\eqref{m3} is new up to our knowledge.

\subsection{Key ideas}

The analysis of \eqref{m1}--\eqref{bc} is based on two observations. First, equations \eqref{m1}--\eqref{m3} can be written as the gradient-flow system
$$
  \pa_t u = \diver(B(u)\na\delta E(u)), \quad\mbox{where }
  u = \begin{pmatrix} \phi \\ c \end{pmatrix},
$$
the so-called mobility matrix reads as
$$
  B(u) = \begin{pmatrix}
  M(\phi) & -cM(\phi) \\ -cM(\phi) & ce^{-\phi} + c^2M(\phi)
  \end{pmatrix},
$$
and $\delta E(u)=(\mu,\pa_c f(\phi,c))^T$ is the variational derivative of the free energy
\begin{equation}\label{1.E}
  E(\phi,c) = \int_\Omega\bigg(\frac12|\na\phi|^2 + f(\phi,c)\bigg)dx,
\end{equation}
which is the sum of the correlation and interaction-nutrient energies. This implies the energy equality (see Lemma \ref{lem.ener})
\begin{equation}\label{1.ei}
  \frac{dE}{dt}(\phi,c)
  + \int_\Omega M(\phi)|\na\mu-c\na\pa_c f(\phi,c)|^2 dx
  + \int_\Omega ce^{-\phi}|\na\pa_c f(\phi,c)|^2 dx = 0,
\end{equation}
yielding $L^2(\Omega)$ bounds for $\na\phi$ and $c$. However, due to the strong coupling, we cannot conclude gradient bounds for the solute concentration $c$. Moreover, the equations are degenerate at $\phi=0$ and $c=0$.

Like in the Cahn--Hilliard equation, we exploit the entropy density
\begin{equation}\label{1.ent}
  \Phi(\phi) = \int_{1/2}^\phi\int_{1/2}^s \frac{drds}{M(r)} \ge 0.
\end{equation}
A formal computation (see Proposition \ref{entropy}) shows that
\begin{align*}
  \frac{d}{dt}&\int_\Omega\Phi(\phi)dx
  + \int_\Omega(\Delta\phi)^2 dx \\
  &= -\int_\Omega\big(\pa_{\phi\phi}^2 f(\phi,c)|\na\phi|^2
  + \pa_{\phi c}^2 f(\phi,c)\na c\cdot\na\phi
  + c\na\pa_c f(\phi,c)\cdot\na\phi\big)dx.
\end{align*}
While the first term on the right-hand side can be bounded by the $L^2(\Omega)$ estimate for $\na\phi$ from \eqref{1.ei} and the last term can be controlled with the help of the energy dissipation in \eqref{1.ei} (if $0\le\phi\le 1$), the second term is more delicate because of the gradient $\na c$. Since $\pa_{\phi c}^2 f(\phi,c)$ is constant in our model, an integration by parts leads to
\begin{align*}
  -\int_\Omega &\pa_{\phi c}^2 f(\phi,c)\na c\cdot\na\phi dx
  = \pa_{\phi c}^2 f\int_\Omega c\Delta\phi dx \\
  &\le C\|c\|_{L^2(\Omega)}\|\Delta\phi\|_{L^2(\Omega)}
  \le C + \frac14\int_\Omega(\Delta\phi)^2 dx.
\end{align*}
Thus, the entropy is not a Lyapunov functional along solutions to \eqref{m1}--\eqref{m3} (as it is in the Cahn--Hilliard system) but it is bounded from above. This yields an $H^2(\Omega)$ bound for $\phi$.

Still, we are missing a gradient estimate for $c$. This is achieved by observing that, as long as $0\le\phi\le 1$,
\begin{align*}
  \int_\Omega c|\na c|^2 dx
  &\le 2\int_\Omega c|\na c-\na\phi|^2 dx
  + 2\int_\Omega c|\na\phi|^2 dx \\
  &\le C\int_\Omega ce^{-\phi}|\na\pa_c f(\phi,c)|^2 dx
  + 2\|c\|_{L^2(\Omega)}\|\na\phi\|_{L^4(\Omega)}^2 \le C,
\end{align*}
since $H^2(\Omega)\hookrightarrow W^{1,4}(\Omega)$ in up to four space dimensions.

These arguments provide gradient bounds for $\phi$ and $c^{3/2}$ under the condition that $0\le\phi\le 1$. These lower and upper bounds cannot be easily derived from the boundedness-by-entropy method \cite{Jue15}, since the relation between the entropy variable $\mu$ and the fiber phase $\phi$ is not algebraic. One idea is based on the minimization of a functional on the set of functions satisfying $0\le\phi\le 1$ \cite{EMP21}. Here, we introduce the entropy with cutoff and conclude the bounds in the limit of vanishing cutoff parameters as in \cite{ElGa96,HJT23}. The idea is simple: The entropy with cutoff $\Phi_\delta$ for $\delta>0$ satisfies
$$
  \Phi_\delta(\phi_\delta(t,x))
  = \frac{(\phi_\delta(t,x)-1/2)^2}{2\delta^2(1-\delta)^2}
  \quad\mbox{for }(t,x)\in
  V_\alpha:=\{\phi_\delta(t,x)\ge 1+\alpha\},
$$
for any $\alpha>0$, where $\phi_\delta$ is the approximate fiber phase fraction (see Section \ref{sec.N}). Therefore, in view of the entropy bound,
$$
  \frac{\alpha^2\mbox{meas}(V_\alpha)}{2\delta^2(1-\delta)^2}
  \le \int_{V_\alpha}\frac{(\phi_\delta-1/2)^2}{2\delta^2(1-\delta)^2}dx
  = \int_{V_\alpha}\Phi_\delta(\phi_\delta)dx \le C.
$$
We obtain $\mbox{meas}(V_\alpha)=0$ in the limit $\delta\to 0$, and since $\alpha>0$ is arbitrary (and using Fatou's lemma), $\phi(t,x)=\lim_{\delta\to 0}\phi_\delta(t,x)<1$ for all $(t,x)$; see Lemma \ref{lemma4} for the precise argument. The lower bound $\phi>0$ is proved in a similar way.

\subsection{Main result}

We first define our notion of weak solution.

\begin{definition}[Weak solution]\label{def}
Let $T>0$ be arbitrary and set $\Omega_T:=(0,T)\times\Omega$. The function $(\phi,c)$ is called a {\em weak solution} to problem \eqref{m1}--\eqref{bc} on $[0,T]$ if $(\phi,c)$ satisfies $0<\phi<1$, $c\ge 0$ in $\Omega_T$,
\begin{align*}
  & \phi\in L^\infty(0,T;H^1(\Omega))\cap L^2(0,T;H^2(\Omega)), \quad
  \pa_t\phi\in L^2(0,T;H^{1}(\Omega)'), \\
  & c\in L^\infty(0,T;L^2(\Omega)),\ c^{3/2}\in L^2(0,T;H^1(\Omega)),\
  \pa_t c\in L^4(0,T;W^{1,8d/(d+4)}(\Omega)'),
\end{align*}
the initial conditions $\phi(0)=\phi_0$ in $L^2(\Omega)$,
$c(0)=c_0$ in the sense of $H^1(\Omega)'$, $(\phi,c)$ verifies the weak formulation
\begin{align*}
  & \int_0^T\langle\partial_t\phi,\psi_1\rangle_1dt
  + \int_0^T\int_\Omega\sqrt{M(\phi)}I\cdot\nabla\psi_1dxdt = 0, \\
  & \int_0^T\langle\partial_t c,\psi_2\rangle_2 dt
  = \int_0^T\int_\Omega c\sqrt{M(\phi)}I\cdot\nabla\psi_2dxdt
  - \int^T_0\int_\Omega ce^{-\phi}\na\pa_cf(\phi,c)\cdot\na\psi_2dxdt
\end{align*}
for all $\psi_1\in L^2(0,T;H^1(\Omega))$, $\psi_2\in L^4(0,T;W^{1,8d/(d+4)}(\Omega))$, and $\mu$ solves
$$
  \mu = -\Delta \phi+\partial_\phi f(\phi,c)\quad\mbox{a.e.\ in }
  \Omega_T.
$$
Here, $\langle\cdot,\cdot\rangle_1$ is the dual product between $H^1(\Omega)'$ and $H^1(\Omega)$, $\langle\cdot,\cdot\rangle_2$ is the dual product between $W^{1,8d/(d+4)}(\Omega)'$ and $W^{1,8d/(d+4)}(\Omega)$, and it holds that
$$
  I=\sqrt{M(\phi)}\big(\nabla\mu-c\nabla\partial_cf(\phi,c)\big)
  \in L^2(\Omega_T)\quad\mbox{in the weak sense},
$$
i.e.\ for any $\Psi\in L^2(0,T;H^{1}(\Omega;\R^d))\cap L^\infty(\Omega_T;\R^d)$ with $\Psi\cdot \nu=0$ on $(0,T)\times\partial\Omega$,
\begin{align*}
  \int_0^T\int_\Omega I\cdot\Psi dxdt
  &= \int_0^T\int_\Omega\big[\Delta\phi\diver\big(\sqrt{M(\phi)}\Psi\big)
  + \big(\nabla\phi-2\sqrt{M(\phi)}\nabla\phi\big)\cdot\Psi \\
  &\phantom{xx}+ c\diver\big(\sqrt{M(\phi)}\Psi\big)
  - \sqrt{M(\phi)}c\nabla\partial_cf(\phi,c)\cdot\Psi\big]dxdt.
\end{align*}
\end{definition}

Since it is difficult to derive an upper bound for $c$, we cannot expect that the weak formulation of the equation for $c$ holds for $\psi_2\in L^2(0,T;H^1(\Omega))$. Because of the degeneracy, we cannot expect a gradient bound for $\mu$, but we obtain $\sqrt{M(\phi)}\na\mu\in L^2(0,T;L^{4/3}(\Omega))$.

Our existence result reads as follows.

\begin{theorem}\label{main}
Let $\Omega\subset\R^d$ $(d\le 4)$ be a bounded domain with boundary $\pa\Omega\in C^2$, let $T>0$, and let $\phi_0\in H^1(\Omega)$, $c_0\in L^2(\Omega)$ satisfy $0<\phi_*\le\phi_0\le 1-\phi_*<1$, $c_0\ge 0$ in $\Omega$ for some $\phi_*\in(0,1)$. Then problem \eqref{m1}--\eqref{bc} possesses a weak solution $(\phi,c)$ in $[0,T]$ in the sense of Definition \ref{def}.
\end{theorem}

The proof is based on an approximation procedure. We introduce three approximation levels. First, we remove the degeneracy in $M(\phi)$ and the singularities in $f(\phi,c)$ by suitable cutoffs with parameter $\delta>0$ and then truncate the diffusion coefficients in the equation for the solute concentration $c$ with parameter $\eps>0$. Because of the lack of a gradient estimate for $c$, we add an artificial diffusion in the equation for $c$ (with parameter $\delta$). Finally, we solve the approximate system in a Faedo--Galerkin space with dimension $N\in\N$. Its global solvability follows from standard arguments and the energy inequality. 
After the limit $N\to\infty$, we are able to conclude the nonnegativity of the concentration. Then the entropy estimate and the artificial diffusion yield gradient bounds and we can pass to the limit $\eps\to 0$. The most delicate part is the limit $\delta\to 0$ in the approximate system with solutions $(\phi_\delta,c_\delta)$. The main idea is to derive a gradient bound for $c_\delta^{3/2}$ and to apply the ``nonlinear'' Aubin--Lions lemma in the version of \cite{CJL14} to conclude the relative compactness of $(c_\delta)$ in $L^3(\Omega)$. To identify the weak limit $I$, we proceed similarly as in \cite[Section 3]{ElGa96}.

The paper is organized as follows. We derive the formal energy and entropy estimates in Section \ref{sec.form}. The Faedo--Galerkin method and the limit $N\to\infty$ are presented in Section \ref{sec.N}. The limits $\eps\to 0$ and then $\delta\to 0$ are performed in Sections \ref{limitvare} and \ref{l}, respectively.


\section{Formal estimations}\label{sec.form}

For the convenience of the reader, we derive the energy and entropy estimates formally for general energy densities $f(\phi,c)$. Rigorous proofs on an approximate level are shown in Lemma \ref{lem.ener} (energy inequality) and Lemma \ref{lemma5} (entropy inequality).

\begin{proposition}[Energy equality]\label{energy}
Let $(\phi,c,\mu)$ be a smooth solution to \eqref{m1}--\eqref{m3} with the initial and boundary conditions \eqref{ic}--\eqref{bc}, satisfying $0\le\phi\le 1$ in $\Omega_T$. Then
\begin{align*}
  \frac{d}{dt}E(\phi,c) + \int_\Omega M(\phi)
  |\na\mu - c\na\pa_c f(\phi,c)|^2 dx
  + \int_\Omega c e^{-\phi}|\na\pa_c f(\phi,c)|^2 dx = 0,
\end{align*}
recalling definition \eqref{1.E} of the energy $E(\phi,c)$.
If $f(\phi,c)\ge\kappa c^2$ holds for some $\kappa>0$, this gives a priori estimates for $\phi$ in $L^\infty(0,T;H^1(\Omega))$ and for $c$ in $L^\infty(0,T;L^2(\Omega))$.
\end{proposition}

\begin{proof}
We compute the time derivative, integrate by parts, and insert equations \eqref{m1} and \eqref{m2}:
\begin{align*}
  \frac{dE}{dt} &= \int_\Omega\big(-\Delta\phi\pa_t\phi
  + \pa_\phi f(\phi,c)\pa_t\phi + \pa_c f(\phi,c)\pa_t c\big)dx
  = \int_\Omega\big(\mu\pa_t\phi + \pa_c f(\phi,c)\pa_t c\big)dx \\
  &= -\int_\Omega\big[M(\phi)\na\mu\cdot(\na\mu-c\pa_cf(\phi,c))
  - cM(\phi)\na\pa_c f(\phi,c)\cdot(\na\mu-c\na\pa_c f(\phi,c)) \\
  &\phantom{xx}+ ce^{-\phi}|\na\pa_c f(\phi,c)|^2\big]dx \\
  &= -\int_\Omega M(\phi)|\na\mu-c\na\pa_c f(\phi,c)|^2 dx
  - \int_\Omega ce^{-\phi}|\na\pa_c f(\phi,c)|^2 dx,
\end{align*}
which ends the proof.
\end{proof}

The second estimate involves the entropy density defined in \eqref{1.ent}.

\begin{proposition}[Entropy equality]\label{entropy}
Let $(\phi,c,\mu)$ be a smooth solution to \eqref{m1}--\eqref{m3} with the initial and boundary conditions \eqref{ic}--\eqref{bc}, satisfying $0\le\phi\le 1$ in $\Omega_T$. We assume that $f(\phi,c)\ge\kappa c^2$ for some $\kappa>0$, $\pa_{\phi\phi}^2 f(\phi,c)$ is bounded from below, 
and $\pa_{\phi c}^2 f(\phi,c)$ is constant for all $(\phi,c)$. Then
\begin{align*}
 \frac{d}{dt}\int_\Omega\Phi(\phi)dx + \frac12\int_\Omega
  (\Delta\phi)^2 dx \le C,
\end{align*}
where $C>0$ depends on $E(\phi_0,c_0)$.
\end{proposition}

\begin{proof}
We compute the time derivative, observe that $\Phi''(\phi)=1/M(\phi)$, and insert equation \eqref{m3} for $\mu$:
\begin{align}\label{2.1}
  \frac{d}{dt}\int_\Omega\Phi(\phi)dx
  &= -\int_\Omega\Phi''(\phi)\na\phi\cdot\big(M(\phi)
  (\na\mu-c\na\pa_c f(\phi,c))\big)dx \\
  &= -\int_\Omega\big(-\na\Delta\phi + \na\pa_\phi f(\phi,c)
  - c\na\pa_c f(\phi,c)\big)\cdot\na\phi dx \nonumber \\
  &= -\int_\Omega\big((\Delta\phi)^2 + \pa_{\phi\phi}^2 f(\phi,c)
  |\na\phi|^2 + \pa_{\phi c}^2 f(\phi,c)\na c\cdot\na\phi \nonumber \\
  &\phantom{xx}- c\na\pa_c f(\phi,c)\cdot\na\phi\big) dx. \nonumber
\end{align}
The last term can be estimated by using H\"older's inequality:
\begin{align*}
  \int_\Omega &c\na\pa_c f(\phi,c)\cdot\na\phi dx \\
  &\le \|\sqrt{c}\|_{L^\infty(0,T;L^4(\Omega))}
  \|\sqrt{c}\na\pa_c f(\phi,c)\|_{L^2(\Omega_T)}
  \|\na\phi\|_{L^2(0,T;L^4(\Omega))}.
\end{align*}
The continuous embedding $H^1(\Omega)\hookrightarrow L^4(\Omega)$ and the elliptic regularity \cite[Theorem 2.24]{Tro87}
\begin{align*}
  \|\phi\|_{H^2(\Omega)}
  \le C\big(\|\Delta\phi\|_{L^2(\Omega)} + \|\phi\|_{H^1(\Omega)}\big)
  \le C\big(\|\Delta\phi\|_{L^2(\Omega)}
  + \|\na\phi\|_{L^2(\Omega)} + 1\big)
\end{align*}
where we also used the bound $0\le\phi\le 1$,
show that
$$
  \|\na\phi\|_{L^4(\Omega)} \le C\|\phi\|_{H^2(\Omega)}
  \le C\big(\|\Delta\phi\|_{L^2(\Omega)}
  + \|\na\phi\|_{L^2(\Omega)} + 1\big).
$$
We infer from Young's inequality and Proposition \ref{energy} that
\begin{align*}
  -\int_\Omega c\na\pa_c f(\phi,c)\cdot\na\phi dx
  \le \frac14\int_\Omega(\Delta\phi)^2 dx
  + \frac14\int_\Omega|\na\phi|^2 dx + C.
\end{align*}
The first term on the right-hand side is absorbed from \eqref{2.1}, while the second one is bounded uniformly in time, because of the energy bound. 

The delicate term is the third term on the right-hand side of \eqref{2.1}, since we do not have any gradient control for $c$. Under the condition that $\pa_{\phi c}^2 f(\phi,c)$ is constant, an integration by parts and Young's inequality yield
\begin{align*}
  -\int_\Omega \pa_{\phi c}^2 f(\phi,c)\na c\cdot\na\phi dx
  = \pa_{\phi c}^2 f\int_\Omega c\Delta\phi dx
  \le C\|c\|_{L^\infty(0,T;L^2(\Omega))}^2
  + \frac14\int_\Omega(\Delta\phi)^2 dx.
\end{align*}
The first term on the right-hand side is controlled by the energy equality, while the second one is absorbed by the first term on the right-hand side of \eqref{2.1}.
\end{proof}

The condition that $\pa_{\phi c}^2f$ is constant can be replaced by (i) $\pa_{\phi c}^2 f$ does not depend on $c$ and is bounded, and (ii) $c\pa_{\phi\phi c}^3f\le 0$. Indeed, by integrating by parts,
\begin{align*}
  -\int_\Omega \pa_{\phi c}^2 f(\phi,c)\na c\cdot\na\phi dx
  &= \int_\Omega\big(c\pa_{\phi\phi c}^3 f|\na\phi|^2
  + c\pa_{\phi c}^2 f\Delta\phi\big)dx \\
  &\le C\|c\|_{L^\infty(0,T;L^2(\Omega))}^2
  \|\pa_{\phi c}^2 f\|_{L^\infty(\Omega_T)}^2
  + \frac14\int_\Omega(\Delta\phi)^2 dx.
\end{align*}
In any case, we need that the potential $\pa_\phi f$ grows at most linearly in $c$ and that the energy $f$ contains a superlinear term in $c$. Thus, we may generalize the nutrient part of the energy, but its structure needs to be similar to the function $(\phi,c)\mapsto c^2/2+c(1-\phi)$ considered in this paper.


\section{Approximate solutions}\label{sec.N}

We introduce first the approximate problem. To this end, we define for fixed $\delta>0$ the nondegenerate mobility
\begin{align*}
  M_\delta(\phi) = \begin{cases}
  M(\delta) &\mbox{if }\phi\leq\delta, \\
  M(\phi) &\mbox{if }\delta\leq\phi\leq 1-\delta, \\
  M(1-\delta) &\mbox{if }\phi\geq 1-\delta.
  \end{cases}
\end{align*}
The free energy density \eqref{1.f} is split into a convex part $f_1$ and a nonconvex part $f_2$, where
$$
  f_1(\phi) = \phi\log\phi + (1-\phi)\log(1-\phi), \quad
  f_2(\phi,c) = \phi(1-\phi) + \frac{c^2}{2} + c(1-\phi).
$$
We define an approximation of $f_1$ on $\R$ to remove the singularities at $\phi=0$ and $\phi=1$:
\begin{align*}
  f_{1,\delta}(\phi) = \begin{cases}
  & f_1(\delta) + f_1'(\delta)(\phi-\delta) + \frac{1}{2}f_1''(\delta)(\phi-\delta)^2
  \quad\mbox{if }\phi\leq\delta, \\
  & f_1(\phi)
  \quad\mbox{if }\delta\leq\phi\leq 1-\delta, \\
  & f_1(1-\delta) + f_1'(1-\delta)(\phi-(1-\delta)) + \frac{1}{2}f_1''(1-\delta)(\phi-(1-\delta))^2
  \quad\mbox{else},
  \end{cases}
\end{align*}
and we set
$$
  f_\delta(\phi,c) = f_{1,\delta}(\phi) + f_2(\phi,c).
$$
Finally, we introduce the truncations
$$
  [\phi]_+^1 = \min\{1,\max\{0,\phi\}\}, \quad
  [c]_+^\eps = \min\{1/\eps,\max\{0,c\}\},
$$
where $0<\eps<1$. Then our approximate system reads as
\begin{align}
  \partial_t\phi &= \diver\big(M_\delta(\phi)(\nabla\mu-[c]_+^\eps
  \nabla\partial_cf_\delta(\phi,c)) \big),\label{1} \\
  \partial_t c &= -{\rm{div}}\big([c]_+^\eps
  M_\delta(\phi)(\nabla\mu-[c]_+^\eps\nabla
  \partial_cf_\delta(\phi,c)) \big)
  + \diver\big([c]_+^\eps e^{-[\phi]^1_+}\nabla
  \partial_cf_\delta(\phi,c) \big)+\delta\Delta c, \label{2} \\
  \mu &= -\Delta\phi+\partial_\phi f_\delta(\phi,c), \nonumber 
\end{align}
with the initial and homogeneous Neumann boundary conditions
\begin{align}
  \phi(0)=\phi_0,\quad c(0)=c_0 &\quad\mbox{in } \Omega, \label{a1} \\
  \nabla\phi\cdot\nu = \nabla\mu\cdot\nu = \nabla c\cdot\nu = 0 &\quad \mbox{on }(0,T)\times\partial\Omega.\label{a2}
\end{align}
Observe that we added the artificial diffusion $\delta\Delta c$ in \eqref{2} to obtain gradient bounds for $c$. The truncations $[c]_+^\eps$ and $[\phi]_+^1$ provide bounded diffusion coefficients needed to derive gradient bounds for the Faedo--Galerkin approximation. The existence of global weak solutions to the nondegenerate approximate problem \eqref{1}--\eqref{a2} is shown in a classical way by means of the Faedo--Galerkin method, energy estimates, and a compactness argument. Since the estimations are strongly model-dependent, we present a full proof.

\subsection{Faedo--Galerkin approximation}

Let $(e_i)_{i\in\N}$ be a complete orthonormal set of eigenfunctions of the Laplacian with homogeneous Neumann boundary conditions in $L^2(\Omega)$. We assume that $e_1=\mbox{const}$, and we set $X_N=\mbox{span}\{e_1,\ldots,e_N\}$ for $N\in\N$. We wish to find solutions $\phi_N$, $c_N$, and $\mu_N$ of the form
$$
  \phi_N(t,x) = \sum_{i=1}^N A_i(t)e_i(x), \quad
  c_N(t,x) = \sum_{i=1}^N B_i(t)e_i(x), \quad
  \mu_N(t,x) = \sum_{i=1}^N C_i(t)e_i(x)
$$
such that
$$
  \phi_N(0)=\sum_{i=1}^N(\phi_0,e_i)_{L^2(\Omega)}e_i,\quad c_N(0)=\sum_{i=1}^N(c_0,e_i)_{L^2(\Omega)}e_i,
$$
and for any $e\in X_N$,
\begin{align}
  \int_{\Omega}\partial_t\phi_N edx
  &= -\int_\Omega M_\delta(\phi_N)(\nabla \mu_N-[c_N]_+^\eps
  \nabla \partial_c f_\delta(\phi_N,c_N))\cdot\nabla edx, \label{a13} \\
  \int_\Omega \partial_t c_N edx
  &= \int_\Omega [c_N]_+^\eps M_\delta(\phi_N)(\nabla \mu_N-[c_N]_+^\eps\nabla \partial_c f_\delta(\phi_N,c_N))\cdot
  \nabla edx \label{a14} \\
  &\phantom{xx} -\int_\Omega [c_N]_+^\eps e^{-[\phi_N]_+^1}
  \nabla\partial_c f_\delta(\phi_N,c_N)\cdot\nabla edx
  - \delta\int_\Omega \nabla c_N\cdot \nabla edx, \nonumber \\
  \int_\Omega \mu_N edx
  &= \int_\Omega \nabla \phi_N\cdot\nabla edx
  + \int_\Omega \partial_\phi f_\delta(\phi_N,c_N)edx.\label{a15}
\end{align}
This means that we wish to find the coefficients $A_i$ and $B_i$, satisfying for $i=1,\ldots,N$ the ordinary differential equations
\begin{align}
  \frac{dA_i}{dt} &= -\int_\Omega M_\delta(\phi_N)(\nabla \mu_N-[c_N]_+^\eps\nabla \partial_c f_\delta(\phi_N,c_N))
  \cdot\nabla e_idx, \label{a3} \\
  \frac{dB_i}{dt} &= \int_\Omega [c_N]_+^\eps
  M_\delta(\phi_N)(\nabla \mu_N-[c_N]_+^\eps\nabla \partial_c f_\delta(\phi_N,c_N))\cdot \nabla e_idx \label{a4} \\
  &\phantom{xx}- \int_\Omega [c_N]_+^\eps
  e^{-[\phi_N]_+^1}\nabla\partial_c f_\delta(\phi_N,c_N)
  \cdot\nabla e_idx -\delta\int_\Omega \nabla c_N\cdot
  \nabla e_i dx, \nonumber \\
  C_i(t) &= \int_\Omega \nabla \phi_N\cdot\nabla
  e_idx + \int_\Omega \partial_\phi f_\delta(\phi_N,c_N)e_idx,
  \nonumber 
\end{align}
and the initial conditions
\begin{align}
  A_i(0)=(\phi_0,e_i)_{L^2(\Omega)},\quad
  B_i(0)=(c_0,e_i)_{L^2(\Omega)}. \label{a6}
\end{align}
Since the right-hand side of system \eqref{a3}--\eqref{a4} depends continuously on $A_i$ and $B_i$ for $i=1,\ldots,N$, Peano's theorem implies the existence of a solution to \eqref{a3}--\eqref{a6} on a time interval $[0,T']$ with $T'\le T$. Then problem \eqref{1}--\eqref{a2} possesses an approximate solution $(\phi_N,c_N,\mu_N)$ on $[0,T']$.

\subsection{Uniform estimates independent of $N$}\label{estimateforN}
In order to extend the solution constructed in the previous subsection globally, it is sufficient to show that $\phi_N$ and $c_N$ are bounded in $X_N$ on $[0,T']$.

\begin{lemma}[Energy inequality]\label{lem.ener}
Let $(\phi_N,c_N,\mu_N)\in X_N^3$ be a solution to problem \eqref{1}--\eqref{a2}. Then $(\phi_N,c_N,\mu_N)$ satisfies
\begin{align}\label{4}
  \frac{d}{dt}&\int_\Omega\bigg(\frac{1}{2}|\nabla\phi_N|^2
  + f_\delta(\phi_N,c_N) \Big)dx
  + \int_\Omega M_\delta(\phi_N)\big|\nabla\mu_N
  - [c_N]_+^\eps\nabla\partial_c f_\delta(\phi_N,c_N)\big|^2 dx \\
  & + \int_\Omega [c_N]_+^\eps e^{-[\phi_N]^1_+}
  |\nabla\partial_c f_\delta(\phi_N,c_N)|^2dx
  + \frac{\delta}{2}\int_\Omega |\nabla c_N|^2dx
  \leq\frac{\delta}{2}\int_\Omega |\nabla\phi_N|^2 dx. \nonumber
\end{align}
\end{lemma}

\begin{proof}
We choose the test functions $e=\mu_N$ in \eqref{a13} and $e=\partial_t \phi_N$ in \eqref{a15} to find that
\begin{align}\label{a7}
  \frac{1}{2}\frac{d}{dt}&\int_\Omega |\nabla \phi_N|^2dx
  + \int_\Omega\pa_\phi f_\delta(\phi_N,c_N)\pa_t\phi_Ndx
  = \int_\Omega\mu_N\pa_t\phi_N dx \\
  &= -\int_\Omega M_\delta(\phi_N)\big(|\nabla\mu_N|^2 - [c_N]_+^\eps
  \nabla\partial_c f_\delta(\phi_N,c_N)\cdot\nabla\mu_N \big)dx.
  \nonumber
\end{align}
The test function $e=\partial_cf_\delta(\phi_N,c_N)=c_N+(1-\phi_N)$ in \eqref{a14} yields
\begin{align}\label{a8}
  \int_\Omega & \partial_t c_N \partial_cf_\delta(\phi_N,c_N)dx \\
  &= \int_\Omega \big([c_N]_+^\eps M_\delta(\phi_N)\nabla\mu_N\cdot
  \nabla\partial_cf_\delta(\phi_N,c_N)
  - M_\delta(\phi_N)\big|[c_N]_+^\eps\nabla\partial_c
  f_\delta(\phi_N,c_N)\big|^2 \big)dx \nonumber \\
  &\phantom{xx}- \int_\Omega [c_N]_+^\eps e^{-[\phi_N]_+^1}
  |\nabla\partial_cf_\delta(\phi_N,c_N)|^2dx
  - \delta\int_\Omega \nabla c_N\cdot
  \nabla\partial_c f_\delta(\phi_N,c_N)dx. \nonumber
\end{align}
Summing \eqref{a7} and \eqref{a8}, some terms can be written as a square, and we end up with
\begin{align*}
  \frac{d}{dt}&\int_\Omega\bigg(\frac{1}{2}|\nabla\phi_N|^2
  + f_\delta(\phi_N,c_N) \bigg)dx + \int_\Omega M_\delta(\phi_N)
  \big|\nabla \mu_N-[c_N]_+^\eps\nabla\partial_c
  f_\delta(\phi_N,c_N)\big|^2dx \\
  &\phantom{xx}+ \int_\Omega [c_N]_+^\eps e^{-[\phi_N]^1_+}
  |\nabla\partial_c f_\delta(\phi_N,c_N)|^2dx
  = -\delta\int_\Omega \nabla c_N\cdot \nabla\partial_c
  f_\delta(\phi_N,c_N)dx. \nonumber \\
  &= -\delta\int_\Omega\na c_N\cdot\na(c_N-\phi_N)dx
  \le -\frac{\delta}{2}\int_\Omega|\nabla c_N|^2dx
  + \frac{\delta}{2}\int_\Omega |\nabla\phi_N|^2dx, \nonumber
\end{align*}
where we used Young's inequality in the last step. This finishes the proof.
\end{proof}

The energy inequality \eqref{4} allows us to conclude some a priori estimates.

\begin{lemma}\label{lemma1}{\rm{(Estimates for $\phi_N$ and $c_N$).}}
Let $\delta\leq 1/12$. There exists a constant $C>0$ independent of $N$ such that
\begin{align}
  \big\|\sqrt{M_\delta(\phi_N)}\big|\nabla \mu_N-[c_N]_+^\eps\nabla
  \partial_cf_\delta(\phi_N,c_N)\big|\big\|_{L^2((0,T')\times\Omega)}
  &\leq C, \label{5} \\
  \|\nabla\phi_N\|_{L^\infty(0,T';L^2(\Omega))}
  + \big\|\sqrt{[c_N]_+^\eps}\nabla\partial_c f_\delta(\phi_N,c_N)
  \big\|_{L^2((0,T')\times\Omega)}
  &\leq C, \label{6} \\
  \|\nabla c_N\|_{L^2((0,T')\times\Omega)}
  + \|c_N\|_{L^\infty(0,T';L^2(\Omega))}
  &\leq C. \label{7}
\end{align}
\end{lemma}

\begin{proof}
The estimates basically follow from the energy inequality \eqref{4}, but we need to estimate $f_{1,\delta}(\phi_N)$ and $f_2(\phi_N,c_N)$, as $\phi_N$ and $c_N$ may be negative. Let $\phi_N<\delta$. We divide the term $f_{1,\delta}(\phi_N)$ into two parts. The first part $f_1(\delta)$ is bounded uniformly in $N$ and $\delta$, while a simple calculation shows that the second part $f_1'(\delta)(\phi_N-\delta) +\frac{1}{2}f_1''(\delta)(\phi_N-\delta)^2$ is nonnegative. If $\delta\le\phi_N\le1-\delta$, the term $f_{1,\delta}(\phi_N)=f_1(\phi_N)$ is bounded uniformly in $N$ and $\delta$. 
The case $\phi_N>1-\delta$ can be treated similarly as the case $\phi_N<\delta$. We turn to the estimate of $f_2(\phi_N,c_N)$. The first term on the right-hand side of
$$
  -\int_\Omega f_2(\phi_N,c_N)dx
  = -\int_\Omega\phi_N(1-\phi_N)dx - \int_\Omega c_N(1-\phi_N)dx
$$
is uniformly bounded with respect to $N$, while the second term is estimated according to
$$
  -\int_\Omega c_N(1-\phi_N)dx \le \frac14\int_\Omega c_N^2 dx
  + \int_\Omega(1-\phi_N)^2dx.
$$
The first term on the right-hand side can be treated in view of inequality \eqref{4} by means of Gronwall's inequality. For the last term, we distinguish several cases. If $\delta\le\phi_N\le 1-\delta$, we have $(1-\phi_N)^2\le 1$; if $\phi_N<\delta$, we choose $\delta\le 1/12$ to find that $\delta(1-\delta)\le 1/12$ and consequently,
$$
  (1-\phi_N)^2 \le 2(\delta-\phi_N)^2 + 2(1-\delta)^2
  \le \frac{(\phi_N-\delta)^2}{6\delta(1-\delta)} + 2,
$$
which is bounded in view of $f_1''(\delta)=1/(\delta(1-\delta))$ and the definition of $f_{1,\delta}$; finally, if $\phi_N>1-\delta$, we obtain in a similar way
$$
  (1-\phi_N)^2 \le \frac{(\phi_N-(1-\delta))^2}{6\delta(1-\delta)} + 2.
$$
The terms involving $\phi_N$ can be treated by taking into account $f_{1,\delta}(\phi_N)$ and Gronwall's lemma.
\end{proof}

Lemma \ref{lemma1} provides an $H^1(\Omega)$ bound for $c_N$. We can also derive such bounds for $\phi_N$ and $\mu_N$.

\begin{lemma}[$H^1(\Omega)$ estimates for $\phi_N$ and $\mu_N$]
\label{lemma2}
For $\delta\le 1/12$, there exists a constant $C>0$ independent of $N$ such that
\begin{align}
  \|\phi_N\|_{L^\infty(0,T';H^1(\Omega))}
  + \|\mu_N\|_{L^2(0,T';H^1(\Omega))}\leq C. \label{8}
\end{align}
\end{lemma}

\begin{proof}
The test function $e=1$ in \eqref{a13} yields conservation of the total fiber phase, $\int_\Omega\phi(x,t)dx=\int_\Omega\phi_0(x)dx$ for all $0\le t\le T'$. Then estimate \eqref{6} for $\na\phi_N$ and the Poincar\'e--Wirtinger inequality lead to an $H^1(\Omega)$ bound for $\phi_N$, showing the first claim. Next, we deduce from bounds \eqref{6} and \eqref{7} that
\begin{align*}
  \big\|\sqrt{M_\delta(\phi_N)}&\na\mu_N
  \big\|_{L^2((0,T')\times\Omega)}
  \le \big\|\sqrt{M_\delta(\phi_N)}\big|\nabla \mu_N-[c_N]_+^\eps\nabla
  \partial_cf_\delta(\phi_N,c_N)\big|\big\|_{L^2((0,T')\times\Omega)} \\
  &+ C\big\|\sqrt{[c_N]_+^\eps}\big\|_{L^\infty((0,T')\times\Omega)}
  \big\|\sqrt{[c_N]_+^\eps}\na\pa_c f_\delta(\phi_N,c_N)
  \big\|_{L^2((0,T')\times\Omega)} \le C,
\end{align*}
which gives a uniform $L^2(\Omega)$ bound for $\na\mu_N$, since $M(\phi_N)\ge M(\delta)>0$. The test function $e=1$ in equation \eqref{a15} for $\mu_N$ gives a uniform bound for $\int_\Omega \mu_N dx$ uniformly in time. Therefore, using the Poincar\'e--Wirtinger inequality again shows that $\mu_N$ is uniformly bounded in $L^2(0,T';H^1(\Omega))$, proving the second claim.
\end{proof}

The uniform $L^2(\Omega)$ bounds for $(\phi_N,c_N)$ imply that the approximate solution $(\phi_N,c_N)$ to \eqref{a13}--\eqref{a15} exists on the whole interval $[0,T]$, and estimates \eqref{5}-\eqref{8} hold on that interval. For the limit $N\to\infty$, we need bounds for the time derivatives.

\begin{lemma}[Estimates for the time derivatives]\label{lemma3}
For $\delta\le 1/12$, there exists a constant $C>0$ independent of $N$ such that
\begin{align}\label{a29}
  \|\partial_t\phi_N\|_{L^2(0,T;H^{1}(\Omega)')}
  + \|\partial_tc_N\|_{L^{2}(0,T;H^{1}(\Omega)')} \leq C.
\end{align}
\end{lemma}

\begin{proof}
Let $\Pi_N$ denote the projection of $L^2(\Omega)$ onto $X_N=\mbox{span}\{e_1,\ldots,e_N\}$.
Based on the estimates obtained in Lemmas \ref{lemma1}--\ref{lemma2}, we have, for any $\psi_1\in L^2(0,T;H^1(\Omega))$,
\begin{align*}
  \bigg|\int_0^T&\int_\Omega \partial_t\phi_N\psi_1dxdt\bigg|
  = \bigg|\int_0^T\int_\Omega \partial_t\phi_N\Pi_N\psi_1dxdt\bigg| \\
  &\le \big\|\sqrt{M_\delta(\phi_N)}\big\|_{L^\infty(\Omega_T)}
  \big\|\sqrt{M_\delta(\phi_N)}
  \big|\nabla \mu_N-[c_N]_+^\varepsilon\nabla \partial_c f_\delta(\phi_N,c_N)\big|\big\|_{L^2(\Omega_T)} \\
  &\phantom{xx}\times\|\nabla \Pi_N\psi_1\|_{L^2(\Omega_T)}
  \leq C\|\psi_1\|_{L^2(0,T;H^1(\Omega))},
\end{align*}
which proves the first claim. Similarly, for any $\psi_2\in L^2(0,T;H^{1}(\Omega))$,
\begin{align*}
  \bigg|\int_0^T&\int_\Omega \partial_tc_N\psi_2dxdt\bigg|
  = \bigg|\int_0^T\int_\Omega \partial_tc_N\Pi_N\psi_2dxdt\bigg| \\
  &\leq \|[c_N]_+^\varepsilon\|_{L^\infty(\Omega_T)} \big\|\sqrt{M_\delta(\phi_N)}\big\|_{L^\infty(\Omega_T)}
  \|\nabla\Pi_N\psi_2\|_{L^2(\Omega_T)}\\
  &\phantom{xx}\times\big\|\sqrt{M_\delta(\phi_N)}\big|\nabla \mu_N-[c_N]_+^\varepsilon\nabla \partial_c f_\delta(\phi_N,c_N)\big|\big\|_{L^2(\Omega_T)} \\
  &\phantom{xx}+\big\|\sqrt{[c_N]_+^\varepsilon}\big\|_{L^\infty(\Omega_T)} \|e^{-[\phi_N]_+^1}\|_{L^\infty(\Omega_T)}
  \|\nabla \Pi_N\psi_2\|_{L^2(\Omega_T)} \\
  &\phantom{xx} + \big\|\sqrt{[c_N]_+^\varepsilon}\nabla\partial_c f_\delta(\phi_N,c_N)\big\|_{L^2(\Omega_T)} + \delta\| \nabla c_N\|_{L^2(\Omega_T)}\| \nabla \Pi_N\psi_2\|_{L^2(\Omega_T)} \\
  &\leq C\|\psi_2\|_{L^2(0,T;H^{1}(\Omega))}.
\end{align*}
This finishes the proof of Lemma \ref{lemma3}.
\end{proof}


\subsection{Compactness argument}\label{limitN}

Estimates \eqref{7} and \eqref{8} allow us to extract subsequences $(\phi_N,c_N,\mu_N)$ (not relabeled) such that as $N\to\infty$, for some functions $\phi$, $c$ and $\mu$,
\begin{align}
  \phi_N\stackrel{*}{\rightharpoonup} \phi &\quad\mbox{weakly in }
  L^\infty(0,T;H^1(\Omega)), \label{9} \\
  c_N\rightharpoonup  c &\quad\mbox{weakly in } L^2(0,T;H^1(\Omega)),
  \label{10} \\
  \mu_N\rightharpoonup \mu &\quad\mbox{weakly in } L^2(0,T;H^1(\Omega)).
  \label{a21}
\end{align}
The estimates for the time derivatives \eqref{a29} yield
\begin{align*}
  \partial_t\phi_N\rightharpoonup \partial_t\phi, \quad
  \partial_tc_N\rightharpoonup \partial_tc
  \quad\mbox{weakly in }L^2(0,T;H^{1}(\Omega)'). 
\end{align*}
We deduce from the Aubin--Lions compactness lemma, in view of estimates \eqref{7}, \eqref{8}, and \eqref{a29}, that, up to a subsequence,
\begin{align}
  \phi_N\rightarrow \phi &\quad\mbox{strongly in }
  C([0,T];L^2(\Omega)), \label{11} \\
  c_N\rightarrow c &\quad\mbox{strongly in }
  L^2(\Omega_T)\cap C([0,T];H^1(\Omega)'). \label{12}
\end{align}
It follows from \eqref{11} and the continuity of $M_\delta(\cdot)$ that
$M_\delta(\phi_N)\rightarrow M_\delta(\phi)$ a.e.\ in $\Omega_T$.
Thus, in accordance with the bound $\|M_\delta(\phi_N)\|_{L^\infty(\Omega_T)}\leq C$,
we infer that
\begin{align}
  M_\delta(\phi_N)\rightarrow M_\delta(\phi)\quad\mbox{strongly in }
  L^p(\Omega_T)\mbox{ for any }p\in[1,\infty). \label{13}
\end{align}
Similar arguments give
\begin{align}
  e^{-[\phi_N]_+^1}\rightarrow e^{-[\phi]_+^1},\quad
  [c_N]_+^\varepsilon\rightarrow [c]_+^\varepsilon
  \quad\mbox{strongly in }L^p(\Omega_T)\mbox{ for any } p\in[1,\infty). \label{a31}
\end{align}
Then it follows from \eqref{5}, \eqref{9}-\eqref{a21}, \eqref{13}, and \eqref{a31} that
\begin{align*}
  M_\delta(&\phi_N)\big(\nabla \mu_N-[c_N]_+^\varepsilon
  (\nabla c_N-\nabla \phi_N) \big) \\
  &\rightharpoonup M_\delta(\phi)\big(\nabla \mu-[c]_+^\varepsilon
  (\nabla c-\nabla \phi) \big)\quad\mbox{weakly in } L^2(\Omega_T). \nonumber
\end{align*}

The partial derivative $\pa_\phi f_\delta$ is continuous in $\R^2$, thanks to the cutoff, and we conclude from \eqref{11}-\eqref{12} that
\begin{align*}
  \partial_\phi f_\delta(\phi_N,c_N)\rightarrow
  \partial_\phi f_\delta(\phi,c)\quad\mbox{a.e.\ in }\Omega_T.
\end{align*}
Taking into account the interpolation inequality in Lebesgue spaces, bounds \eqref{7} for $c_N$, and the continuous embedding $H^1(\Omega)\hookrightarrow L^{2d/(d-2)}(\Omega)$, we find that
\begin{align}\label{a11}
  \|c_N\|_{L^4(0,T;L^{2d/(d-1)}(\Omega))}
  &\le \bigg(\int_0^T\|c_N\|_{L^2(\Omega)}^2
  \|c_N\|_{L^{2d/(d-2)}(\Omega)}^2 dt\bigg)^{1/4} \\
  &\leq \|c_N\|_{L^\infty(0,T;L^2(\Omega))}^{1/2}
  \|c_N\|_{L^2(0,T;L^{2d/(d-2)}(\Omega))}^{1/2}\leq C. \nonumber
\end{align}
By definition of $f_\delta$, this estimate and bound \eqref{8} yield
\begin{align*}
  \|\partial_\phi f_\delta(\phi_N,c_N)
  \|_{L^4(0,T;L^{{2d/(d-2)}}(\Omega))}
  &\leq C\|\phi_N\|_{L^\infty(0,T;L^{2d/(d-2)}(\Omega))} \\
  &\phantom{xx}+ C\|c_N\|_{L^4(0,T;L^{{2d/(d-2)}}(\Omega))} + C
  \leq C.
\end{align*}
Therefore, we achieve the strong convergence
\begin{align*}
  \partial_\phi f_\delta(\phi_N,c_N)\rightarrow \partial_\phi f_\delta(\phi,c) \quad\mbox{strongly in }L^2(\Omega_T).
\end{align*}

By the previous convergence results, we can take the limit $N\to\infty$ in system \eqref{a13}-\eqref{a15} for $(\phi_N,c_N,\mu_N)$, and the limit functions $(\phi,c,\mu)$ solve
\begin{align}
  \int_0^T&\langle\partial_t\phi, \psi_1\rangle_1dt
  = -\int_0^T\int_\Omega M_\delta(\phi)
  (\nabla \mu-[c]_+^\varepsilon\nabla \partial_c
  f_\delta(\phi,c))\cdot\nabla \psi_1 dxdt, \label{a16} \\
  \int_0^T&\langle \partial_t c, \psi_2\rangle_1\,dt
  = \int_0^T\int_\Omega [c]_+^\varepsilon M_\delta(\phi)
  (\nabla \mu-[c]_+^\varepsilon\nabla \partial_c f_\delta(\phi,c))
  \cdot \nabla \psi_2dxdt \label{a17} \\
  &\phantom{xx}- \int_0^T\int_\Omega [c]_+^\varepsilon
  e^{-[\phi]_+^1}\nabla\partial_c f_\delta(\phi,c)\cdot\nabla
  \psi_2dxdt - \delta\int_0^T\int_\Omega \nabla c\cdot \nabla
  \psi_2 dxdt, \nonumber \\
  \int_0^T&\int_\Omega \mu \psi_3dxdt
  = \int_0^T\int_\Omega \nabla \phi\cdot\nabla \psi_3dxdt
  + \int_0^T\int_\Omega \partial_\phi f_\delta(\phi,c)\psi_3dxdt
  \label{a18}
\end{align}
for all test functions $\psi_1$, $\psi_2$, $\psi_3\in L^2(0,T;H^1(\Omega))$, recalling that $\langle\cdot,\cdot\rangle_1$ is the dual product between $H^1(\Omega)'$ and $H^1(\Omega)$.

It remains to show that $\phi$ and $c$ attain the initial conditions.
We deduce from \eqref{11} and the fact that $\phi_N(0)\rightarrow \phi_0$ strongly in $L^2(\Omega)$ that $\phi(0)=\phi_0$ in $\Omega$.
The strong convergence of $c_N$ to $c$ in $C([0,T];H^1(\Omega)')$ implies that $\langle c_N(0),\xi\rangle_1\rightarrow \langle c(0),\xi\rangle_1$
for any $\xi\in H^1(\Omega)$. We combine this result and $c_N(0)\rightarrow c_0$ in $L^2(\Omega)$ to find that $\langle c(0),\xi\rangle_1=\langle c_0,\xi\rangle_1$.


\section{The limit $\varepsilon\rightarrow 0$}\label{limitvare}

In this section, we derive a lower bound for $c$ and perform the limit $\varepsilon\rightarrow 0$ in equations \eqref{a16}-\eqref{a18}.
We denote by $(\phi_\varepsilon,c_\varepsilon)$ the solution at this level of approximation.

\subsection{Uniform estimates independent of $\varepsilon$}

We start with the lower bound for $c_\varepsilon$.

\begin{lemma}[Lower bound for $c_\varepsilon$]
Let $(\phi_\varepsilon,c_\varepsilon)$ be a solution to \eqref{a16}-\eqref{a18}. Then
\begin{align*}
 c_\varepsilon(t,x)\geq 0\quad\mbox{a.e.\ in }\Omega_T.
\end{align*}
\end{lemma}

\begin{proof}
The proof is easy since it is sufficient to take the test function $\psi_2=[c_\varepsilon]_-:=-\min\{0,c_\varepsilon\}$ in \eqref{a17}, which yields
\begin{align*}
  \frac{1}{2}\int_\Omega [c_\varepsilon]_-^2(\tau,x)dx
  + \delta\int_0^\tau\int_\Omega |\nabla[c_\varepsilon]_-|^2dxdt = 0,
\end{align*}
and hence $[c_\varepsilon]_-=0$ a.e.\ in $\Omega_T$, finishing the proof.
\end{proof}

The previous lemma shows that we can replace the truncation $[c_\eps]_+^\eps$ by $[c_\eps]^\eps:=\min\{1/\eps$, $c_\eps\}$, and it remains to remove the upper truncation.

We claim that $(\phi_\eps,c_\eps)$ satisfies an energy inequality similar to \eqref{4}. As a preparation, by elliptic regularity theory, we conclude from  $\Delta\phi_\varepsilon=-\mu_\varepsilon+\partial_\phi f_\delta(\phi_\varepsilon,c_\varepsilon)\in L^2(\Omega_T)$ that
$\phi_\varepsilon\in L^2(0,T;H^2(\Omega))$. Furthermore, it follows from $\nabla\Delta\phi_\varepsilon=-\nabla\mu_\varepsilon+\nabla\partial_\phi f_\delta(\phi_\varepsilon,c_\varepsilon)\in L^2(\Omega_T)$ that $\phi_\varepsilon\in L^2(0,T;H^3(\Omega))$. Therefore, we can compute for $\tau\in(0,T)$,
\begin{align*}
  0 &= \int_0^\tau \langle\partial_t\phi_\varepsilon, \mu_\varepsilon+\Delta\phi_\varepsilon-\partial_\phi f_\delta(\phi_\varepsilon,c_\varepsilon)\rangle_1dt \\
  &= \int_0^\tau \langle\partial_t\phi_\varepsilon, \mu_\varepsilon\rangle_1dt
  - \frac{1}{2}\int_\Omega \big(|\nabla\phi_\varepsilon(\tau,x)|^2
  - |\nabla\phi_0(x)|^2\big)dx
  - \int_0^\tau \langle\partial_t\phi_\varepsilon,
  \partial_\phi f_\delta(\phi_\varepsilon,c_\varepsilon)\rangle_1dt.
\end{align*}
Choosing $\psi_1=\mu_\varepsilon\in L^2(0,T;H^1(\Omega))$ in \eqref{a16} and $\psi_2=\partial_c f_\delta(\phi_\varepsilon,c_\varepsilon)\in L^2(0,T;H^1(\Omega))$ in \eqref{a17}, we obtain the energy equality
\begin{align}\label{15}
  \int_\Omega&\bigg(\frac{1}{2}|\nabla\phi_\varepsilon|^2
  + f_\delta(\phi_\varepsilon,c_\varepsilon) \bigg)(\tau,x)dx
  + \int_0^\tau\int_\Omega M_\delta(\phi_\varepsilon)
  |\nabla \mu_\varepsilon - [c_\varepsilon]^\varepsilon
  \nabla\partial_c f_\delta(\phi_\varepsilon,c_\varepsilon)|^2dxdt \\
  &\phantom{xx} + \int_0^\tau\int_\Omega [c_\varepsilon]^\varepsilon
  e^{-[\phi_\varepsilon]^1_+}|\nabla\partial_c
  f_\delta(\phi_\varepsilon,c_\varepsilon)|^2dxdt
  + \delta\int_0^\tau\int_\Omega |\nabla c_\varepsilon|^2
  dxdt \nonumber \\
  &= \int_\Omega\bigg(\frac{1}{2}|\nabla\phi_0|^2
  + f_\delta(\phi_0,c_0) \bigg)dx
  - \delta\int_0^\tau\int_\Omega \nabla c_\varepsilon
  \cdot\nabla \phi_\varepsilon dxdt. \nonumber
\end{align}
It follows from this equality and similar arguments as the proof of Lemmas \ref{lemma1}--\ref{lemma3}, for $\delta\leq 1/12$, that
\begin{align}
  \big\|\sqrt{M_\delta(\phi_\varepsilon)}\big|\nabla \mu_\varepsilon - [c_\varepsilon]^\varepsilon\nabla\partial_cf_\delta
  (\phi_\varepsilon,c_\varepsilon)\big|\big\|_{L^2(\Omega_T)}
  &\leq C, \label{ad5} \\
  \big\|\sqrt{[c_\varepsilon]^\varepsilon}\nabla\partial_c f_\delta(\phi_\varepsilon,c_\varepsilon)
  \big\|_{L^2(\Omega_T)}
  &\leq C, \label{ad6} \\
  \|\nabla c_\varepsilon\|_{L^2(\Omega_T)}
  + \|c_\varepsilon\|_{L^\infty(0,T;L^2(\Omega))}
  &\leq C, \label{ad7} \\
  \|\phi_\varepsilon\|_{L^\infty(0,T;H^1(\Omega))}
  + \|\partial_t\phi_\varepsilon\|_{L^2(0,T;H^{1}(\Omega)')}
  &\leq C, \label{ad4}
\end{align}
where the constant $C>0$ is independent of $\eps$ (but possibly depending on $\delta$).

In contrast to Section \ref{estimateforN}, we cannot expect a uniform $L^\infty(\Omega)$ bound for $[c_\varepsilon]^\varepsilon$, which makes it necessary to take advantage of \eqref{ad7}. It follows from H\"older's inequality and bounds \eqref{ad6}, \eqref{ad7} that
\begin{align*}
  \big\|[c_\varepsilon]^\varepsilon & \nabla \partial_cf_\delta(\phi_\varepsilon, c_\varepsilon)\big\|_{L^2(0,T;L^{4/3}(\Omega))} \\
  \leq& \big\|\sqrt{c_\varepsilon}\big\|_{L^\infty(0,T;L^4(\Omega))}
  \big\|\sqrt{[c_\varepsilon]^\varepsilon}\nabla
  \partial_cf_\delta(\phi_\varepsilon,c_\varepsilon)
  \big\|_{L^2(\Omega_T)}\leq C.
\end{align*}
Together with \eqref{ad5} and the property $M_\delta(\phi_\varepsilon)\geq c$ for some $c>0$ independent of $\eps$, we infer that $(\na\mu_\eps)$ is bounded in $L^2(0,T;L^{4/3}(\Omega))$. The test function $\psi_3=1$ in \eqref{a18} provides a uniform bound for $|\int_\Omega\mu_\eps dx|$, so by the Poincar\'e--Wirtinger inequality,
\begin{align}\label{ad8}
  \|\mu_\varepsilon\|_{L^2(0,T;W^{1,4/3}(\Omega))}\leq C.
\end{align}

Next, the continuous embedding $H^1(\Omega)\hookrightarrow L^{2d/(d-2)}(\Omega)$ and interpolation as in \eqref{a11} yield
\begin{align}\label{ad11}
  \|c_\eps\|_{L^4(0,T;L^{2d/(d-1)}(\Omega))}
  \leq \|c_\eps\|_{L^\infty(0,T;L^2(\Omega))}^{1/2}
  \|c_\eps\|_{L^2(0,T;L^{2d/(d-2)}(\Omega))}^{1/2}\leq C.
\end{align}
Estimates \eqref{ad5}--\eqref{ad7} allow us to derive a uniform bound for the time derivative $\pa_t c_\eps$. For this, let $\psi\in L^4(0,T;W^{1,2d}(\Omega))$ and estimate as follows:
\begin{align*}
  \bigg|\int_0^T\int_\Omega \partial_tc_\varepsilon\psi dxdt\bigg|
  &\leq \|c_\varepsilon\|_{L^4(0,T;L^{2d/(d-1)}(\Omega))} \big\|\sqrt{M_\delta(\phi_\varepsilon)}
  \big\|_{L^\infty(\Omega_T)}
  \|\nabla\psi\|_{L^4(0,T;L^{2d}(\Omega))} \\
  &\phantom{xxxx}\times\big\|\sqrt{M_\delta(\phi_\varepsilon)}
  \big|\nabla \mu_\varepsilon - [c_\varepsilon]^\varepsilon
  \nabla \partial_c f_\delta(\phi_\varepsilon,c_\varepsilon)\big| \big\|_{L^2(\Omega_T)} \\
  &\phantom{xx} + \big\|\sqrt{c_\varepsilon}
  \big\|_{L^4(0,T;L^{2d/(d-1)}(\Omega))}
  \|e^{-[\phi_\varepsilon]_+^1}\|_{L^\infty(\Omega_T)}
  \|\nabla \psi\|_{L^4(0,T;L^{2d}(\Omega))} \\
  &\phantom{xxxx}\times\big\|\sqrt{[c_\varepsilon]^\varepsilon}
  \nabla\partial_c f_\delta(\phi_\varepsilon,c_\varepsilon)
  \big\|_{L^2(\Omega_T)} \\
  &\phantom{xx} + \delta\| \nabla c_\varepsilon\|_{L^2(\Omega_T)}
  \| \nabla \psi\|_{L^2(\Omega_T)}
  \leq C\|\psi\|_{L^4(0,T;W^{1,2d}(\Omega))}.
\end{align*}
This gives the desired bound for the time derivative:
\begin{align}\label{ad29}
  \|\partial_tc_\varepsilon\|_{L^{4/3}(0,T;W^{1,2d}(\Omega)')}\leq C.
\end{align}

We are now in the position to derive the entropy inequality satisfied by $(\phi_\varepsilon,c_\varepsilon)$. The entropy inequality is needed to pass to the limit $\eps\to 0$ in the energy equality \eqref{15}. We define the approximate entropy density by
\begin{align*}
  \Phi_\delta(\phi_\varepsilon) = \int_{1/2}^{\phi_\varepsilon}
  \int_{1/2}^s \frac{drds}{M_\delta(r)}\geq 0.
\end{align*}

\begin{lemma}[Entropy inequality for $(\phi_\eps,c_\eps)$]\label{lemma5}
There exists a constant $C>0$ independent of $\eps$ such that the following entropy inequality holds for any $\tau\in [0,T]$:
\begin{align}\label{entropyv}
  \int_\Omega \Phi_\delta(\phi_\varepsilon)(\tau,x)dx
  + \frac{1}{2}\int_0^\tau\int_\Omega (\Delta\phi_\varepsilon)^2dxdt
  \leq \int_\Omega \Phi_\delta(\phi_\varepsilon)(0,x)dx + C.
\end{align}
\end{lemma}

\begin{proof}
A similar argument as in the proof of the energy equality \eqref{15} can be used to find that $\Delta\phi_\varepsilon\in L^2(0,T;H^1(\Omega))$. Hence, we can choose the test function $\psi_3=\Delta\phi_\varepsilon$ in \eqref{a18}. It follows from the definition of $\Phi_\delta(\phi_\varepsilon)$ that $\nabla\Phi'_\delta(\phi_\varepsilon)\in L^2(\Omega_T)$.
Therefore, taking $\psi_1=\Phi'_\delta(\phi_\varepsilon)$ in \eqref{a16} and combining this equation with \eqref{a18}, choosing the test function $\psi_3=\Delta\phi_\varepsilon$, we infer, for any $\tau\in [0,T]$, that
\begin{align*}
  \int_\Omega & \Phi_\delta(\phi_\varepsilon)(\tau,x)dx
  + \int_0^\tau\int_\Omega (\Delta \phi_\varepsilon)^2 dxdt \nonumber \\
  &= \int_\Omega \Phi_\delta(\phi_\varepsilon)(0,x)dx
  - \int_0^\tau\int_\Omega \bigg(\frac{\na\phi^\eps}{\sqrt{M_\delta(\phi_\varepsilon)}}
  - \nabla c_\varepsilon-2\nabla\phi_\varepsilon \bigg)\cdot\nabla\phi_\varepsilon dxdt \nonumber \\
  &\phantom{xx} + \int_0^\tau\int_\Omega [c_\varepsilon]^\varepsilon\nabla\partial_c f_\delta(\phi_\varepsilon,c_\varepsilon)\cdot\nabla \phi_\varepsilon dxdt,
\end{align*}
observing that $\Phi''_{\delta}(\phi_\varepsilon) = 1/M_\delta(\phi_\varepsilon)$ and the fact that
\begin{align*}
  \nabla\partial_\phi f_{\delta}(\phi_\varepsilon,c_\varepsilon)
  = \frac{\nabla\phi_\varepsilon}{\sqrt{M_\delta(\phi_\varepsilon)}}
  - \nabla c_\varepsilon-2\nabla\phi_\varepsilon.
\end{align*}
We deduce from $\sqrt{M_\delta(\phi_\varepsilon)} > 0$ that
\begin{align}\label{20}
  \int_\Omega & \Phi_\delta(\phi_\varepsilon)(\tau,x)dx
  + \int_0^\tau\int_\Omega (\Delta \phi_\varepsilon)^2dxdt  \\
  &\leq \int_\Omega \Phi_\delta(\phi_\varepsilon)(0,x)dx
  + \int_0^\tau\int_\Omega \nabla c_\varepsilon\cdot\nabla \phi_\varepsilon dxdt
  + 2\int_0^\tau\int_\Omega |\nabla \phi_\varepsilon|^2dxdt \nonumber \\
  &\phantom{xx}+\int_0^\tau\int_\Omega [c_\varepsilon]^\varepsilon\nabla\partial_c f_\delta(\phi_\varepsilon,c_\varepsilon)\cdot\nabla \phi_\varepsilon dxdt. \nonumber
\end{align}
The second term on the right-hand side can be estimated according to
\begin{align}\label{21}
  \int_0^\tau\int_\Omega \nabla c_\varepsilon\cdot \nabla\phi_\varepsilon dxdt
  = -\int_0^\tau\int_\Omega c_\varepsilon\Delta\phi_\varepsilon dxdt
  \leq C + \frac{1}{4}\int_0^\tau\int_\Omega (\Delta\phi_\varepsilon)^2 dxdt.
\end{align}
Furthermore, in view of bounds \eqref{ad6}-\eqref{ad4} and the continuous embedding of $H^2(\Omega)\hookrightarrow W^{1,4}(\Omega)$ for $d\leq 4$,
the last term in \eqref{20} can be computed as
\begin{align}\label{22}
  \int_0^\tau&\int_\Omega [c_\varepsilon]^\varepsilon\nabla\partial_c
  f_\delta(\phi_\varepsilon,c_\varepsilon)
  \cdot\nabla\phi_\varepsilon dxdt \\
  &\leq \|\sqrt{c_\varepsilon}\|_{L^\infty(0,T;L^4(\Omega))}
  \|\sqrt{[c_\varepsilon]^\varepsilon}\nabla\partial_c f_\delta(\phi_\varepsilon,c_\varepsilon)\|_{L^2(\Omega_T)}
  \|\nabla\phi_\varepsilon\|_{L^2(0,T;L^4(\Omega))} \nonumber \\
  &\leq C(\|\nabla\phi_\varepsilon\|_{L^2(\Omega_T)}
  + \|\Delta \phi_\varepsilon\|_{L^2(\Omega_T)} + 1) \nonumber \\
  &\leq  C +\frac{1}{4}\int_0^T\int_\Omega (\Delta\phi_\varepsilon)^2 dxdt. \nonumber
\end{align}
Inserting \eqref{21}-\eqref{22} into \eqref{20} yields \eqref{entropyv}.
\end{proof}


\subsection{Compactness argument}

The uniform bounds in the previous subsection allow us to perform the limit $\eps\to 0$ in the weak formulation \eqref{a16}--\eqref{a18}. The proof is similar to that one given in Section \ref{limitN}, except the compactness of the terms involving $[c_\varepsilon]^\varepsilon$.
We focus on these terms in the following discussion.

In view of \eqref{ad7}, \eqref{ad8}, and \eqref{ad29}, we can apply the Aubin--Lions compactness lemma to obtain a subsequence $(c_\varepsilon,\mu_\varepsilon)$ (which is not relabeled) and functions $c$ and $\mu$ such that, as $\eps\to 0$,
\begin{align}
  c_\varepsilon\rightarrow c &\quad\mbox{strongly in } L^2(\Omega_T)\cap C([0,T];H^1(\Omega)'), \label{ad12} \\
  \partial_tc_\varepsilon\rightharpoonup \partial_tc
  &\quad\mbox{weakly in }L^{4/3}(0,T;W^{1,2d}(\Omega)'), \label{ad33} \\
  \mu_\varepsilon\rightharpoonup\mu &\quad\mbox{weakly in } L^2(0,T;W^{1,4/3}(\Omega))\cap L^2(\Omega_T). \label{ad20}
\end{align}
It follows from \eqref{ad7} that
\begin{align*}
  \big\|[c_\varepsilon]^\varepsilon-c_\varepsilon
  \big\|_{L^1(\Omega_T)}
  &= \int_0^T\int_\Omega (c_\varepsilon-1/\varepsilon)\chi_{\{c_\varepsilon\geq 1/\varepsilon\}}dxdt
  \leq 2\int_0^T\int_\Omega c_\varepsilon
  \chi_{\{c_\varepsilon\geq 1/\varepsilon \}}dxdt \\
  &\leq 2\varepsilon\int_0^T\int_\Omega c_\varepsilon^2dxdt
  \leq C\varepsilon\rightarrow 0.
\end{align*}
This, together with \eqref{ad12}, implies that
$[c_\varepsilon]^\varepsilon\to c$ a.e.\ in $\Omega_T$. Then
bound \eqref{ad11} shows that
\begin{align}\label{ad31}
  [c_\varepsilon]^\varepsilon\rightarrow c\quad\mbox{strongly in } L^p(0,T;L^q(\Omega))\mbox{ for } p\in[1,4),\ q\in \bigg[1,\frac{2d}{d-1}\bigg).
\end{align}
Since $M_\delta(\phi_\varepsilon)$ converges strongly to $M_\delta(\phi)$ in $L^p(\Omega_T)$ for $1\leq p<\infty$
and $\na\pa_c f(\phi_\eps,c_\eps)=\nabla\phi_\varepsilon-\nabla c_\varepsilon$ converges weakly to $\na\pa_c f(\phi,c)=\nabla\phi-\nabla c$ in $L^2(\Omega_T)$,
it follows from estimate \eqref{ad5} and the convergence results \eqref{ad20} and \eqref{ad31} that
\begin{align}\label{ad14}
  M_\delta(\phi_\varepsilon)&
  \big(\nabla \mu_\varepsilon-[c_\varepsilon]^\varepsilon
  \na\pa_c f(\phi_\eps,c_\eps) \big) \\
  &\rightharpoonup M_\delta(\phi)\big(\nabla \mu-c\na\pa_c f(\phi,c) \big)\quad\mbox{weakly in }L^2(\Omega_T). \nonumber
\end{align}

Taking into account the convergences \eqref{ad12}--\eqref{ad14}, we can pass to the limit $\varepsilon\rightarrow 0$ in system \eqref{a16}--\eqref{a18}, for the variables $(\phi_\varepsilon,c_\varepsilon,\mu_\varepsilon)$, to conclude that the
triplet $(\phi,c,\mu)$ is a weak solution to the following problem:
\begin{align}
  \partial_t\phi &= \diver\big(M_\delta(\phi)(\nabla\mu
  - c\nabla\partial_cf_\delta(\phi,c)) \big), \label{ad16} \\
  \partial_t c &= -{\rm{div}}\big(c M_\delta(\phi) (\nabla\mu-c\nabla\partial_cf_\delta(\phi,c)) \big)
  + \diver\big(c e^{-[\phi]^1_+}\nabla\partial_cf_\delta(\phi,c) \big)
  + \delta\Delta c, \label{ad17} \\
  \mu &= -\Delta\phi + \partial_\phi f_\delta(\phi,c), \label{ad18}
\end{align}
with the boundary and initial conditions
$$
  \nabla\phi\cdot\nu = \nabla c\cdot\nu = \nabla\mu\cdot\nu = 0
  \quad\mbox{on }(0,T)\times\partial\Omega, \quad
  (\phi(0),c(0)) = (\phi_0,c_0)\quad\mbox{in }\Omega.
$$
The weak formulations of \eqref{ad16} and \eqref{ad18} are the same as \eqref{a16} and \eqref{a18}, respectively, while the weak formulation of \eqref{ad17} can be written, for any $\psi\in L^4(0,T;W^{1,2d}(\Omega))$, as
\begin{align}\label{ad19}
  \int_0^T\langle \partial_t c, \psi\rangle_3dt
  &= \int_0^T\int_\Omega c M_\delta(\phi)(\nabla \mu-c\nabla
  \partial_c f_\delta(\phi,c))\cdot \nabla \psi dxdt \\
  &\phantom{xx} - \int_0^T\int_\Omega c e^{-[\phi]_+^1}\nabla
  \partial_c f_\delta(\phi,c)\cdot\nabla \psi dxdt
 - \delta\int_0^T\int_\Omega \nabla c\cdot \nabla \psi dxdt,\nonumber
\end{align}
where $\langle\cdot,\cdot\rangle_3$ denotes the dual product between $W^{1,2d}(\Omega)'$ and $W^{1,2d}(\Omega)$.

As the weak formulation \eqref{ad19} holds for $\psi\in  L^4(0,T;W^{1,2d}(\Omega))$, we cannot choose $\psi=\pa_c f_\delta(\phi,c)$ as a test function in \eqref{ad19} as in the proof of the energy equality \eqref{15}. Fortunately, we can derive this identity, satisfied by $(\phi,c)$, in another way. Thanks to \eqref{ad12}--\eqref{ad14} and the weak lower semicontinuity of convex functions, we are able to pass to the limit $\varepsilon\rightarrow 0$ in \eqref{15}, expect for the last term. The entropy inequality \eqref{entropyv} gives a uniform $L^2(\Omega)$ bound for $\Delta\phi_\varepsilon$, which implies that $\Delta\phi_\varepsilon$ weakly converges to $\Delta\phi$ in $L^2(\Omega)$. We deduce from the strong convergence \eqref{ad12} of $c_\eps$ that
\begin{align*}
  -\int_0^\tau\int_\Omega\na c_\eps\cdot\na\phi_\eps dxdt
  &= \int_0^\tau\int_\Omega c_\varepsilon\Delta\phi_\varepsilon dxdt \\
  &\rightarrow \int_0^\tau\int_\Omega c\Delta\phi dxdt
  = -\int_0^\tau\int_\Omega\na c\cdot\na\phi dxdt.
\end{align*}
As a consequence, $(\phi,c)$ satisfies, for any $\tau\in [0,T]$, the energy equality
\begin{align}\label{ad15}
  \int_\Omega & \bigg(\frac{1}{2}|\nabla\phi|^2
  + f_\delta(\phi,c) \bigg)(\tau,x) dx
  + \int_0^\tau\int_\Omega M_\delta(\phi)|\nabla \mu
  - c\nabla\partial_c f_\delta(\phi,c)|^2 dxdt \\
  &\phantom{xx} + \int_0^\tau\int_\Omega ce^{-[\phi]^1_+}
  |\nabla\partial_c f_\delta(\phi,c)|^2 dxdt
  + \delta\int_0^\tau\int_\Omega |\nabla c|^2 dxdt \nonumber \\
  &= \int_\Omega\bigg(\frac{1}{2}|\nabla\phi_0|^2
  + f_\delta(\phi_0,c_0) \bigg)dx
 - \delta\int_0^\tau\int_\Omega \nabla c\cdot\nabla \phi dxdt. \nonumber
\end{align}


\section{The limit $\delta\rightarrow 0$}\label{l}

For the proof of Theorem \ref{main}, it remains to pass to the limit $\delta\to 0$ in the weak formulation of system \eqref{ad16}--\eqref{ad18}. Let $(\phi_\delta,c_\delta,\mu_\delta)$ be the solution constructed in the previous section.

\subsection{Uniform estimates independent of $\delta$}

By the same argument as that one used in the proof of Lemma \ref{lemma1}, we derive from the energy equality \eqref{ad15} the following estimates:
\begin{align}
  \|\phi_\delta\|_{L^\infty(0,T;H^1(\Omega))}
  + \big\|\sqrt{M_\delta(\phi_\delta)}\big|\nabla \mu_\delta
  - c_\delta\nabla\partial_cf_\delta(\phi_\delta,c_\delta)\big|
  \big\|_{L^2(\Omega_T)}
  &\leq C, \label{16} \\
  \sqrt{\delta}\|\nabla c_\delta\|_{L^2(\Omega_T)}
  + \|c_\delta\|_{L^\infty(0,T;L^2(\Omega))}
  &\leq C, \label{17} \\
  \|\sqrt{c_\delta}\nabla\partial_cf_\delta
  (\phi_\delta,c_\delta)\|_{L^2(\Omega_T)}
  &\leq C, \label{18}
\end{align}
where $C$ is a positive constant independent of $\delta$.

The entropy inequality from Lemma \ref{lemma5} and the above compactness arguments show that $\Delta\phi_\delta\in L^2(\Omega_T)$. Thus, since $\mu_\delta$, $\partial_\phi f_\delta(\phi_\delta,c_\delta)\in L^2(\Omega_T)$,
the weak formulation of \eqref{ad18} holds for any $\psi\in L^2(\Omega_T)$. We can choose $\Delta\phi_\delta$ as a test function in this weak formulation and proceed as in Lemma \ref{lemma5} to obtain, for any $\tau\in [0,T]$, the entropy inequality
\begin{align}\label{19}
  \int_\Omega \Phi_\delta(\phi_\delta)(\tau,x)dx
  + \frac{1}{2}\int_0^\tau\int_\Omega (\Delta\phi_\delta)^2dxdt
  \leq \int_\Omega \Phi_\delta(\phi_\delta)(0,x)dx + C.
\end{align}

\begin{lemma}[Estimates for $\phi_\delta$ and $c_\delta$]
There exists a constant $C>0$ independent of $\delta$ such that
\begin{align}
  \|\phi_\delta\|_{L^2(0,T;H^2(\Omega))}&\leq C, \label{a20} \\
  \big\|c_\delta^{3/2}\big\|_{L^2(0,T;H^1(\Omega))}
  + \|c_\delta\|_{L^4(0,T;L^{8d/(3d-4)}(\Omega))}
  &\leq C. \label{24}
\end{align}
\end{lemma}

\begin{proof}
The entropy inequality \eqref{19} provides a uniform bound for $\Delta\phi_\eps$ in $L^2(\Omega_T)$. This, together with \eqref{16}, implies \eqref{a20}. As a consequence, $\na\phi_\eps$ is uniformly bounded in $L^2(0,T;$ $H^1(\Omega))\hookrightarrow L^2(0,T;L^4(\Omega))$ (for $d\le 4$), which gives
\begin{align*}
  \|\sqrt{c_\delta}\nabla\phi_\delta\|_{L^2(\Omega_T)}
  \leq \|\sqrt{c_\delta}\|_{L^\infty(0,T;L^4(\Omega))}
  \|\nabla\phi_\delta\|_{L^2(0,T;L^4(\Omega))}\leq C.
\end{align*}
We conclude from \eqref{18} that
\begin{align*}
  \big\|\na c_\delta^{3/2}\big\|_{L^2(\Omega_T)}
  &= \frac32\|\sqrt{c_\delta}\na c_\delta
  \|_{L^2(\Omega_T)} \\
  &\le C\|\sqrt{c_\delta}\na\pa_c f(\phi_\delta,c_\delta)
  \|_{L^2(\Omega_T)} + C\|\sqrt{c_\delta}\na\phi_\delta
  \|_{L^2(\Omega_T)} \le C.
\end{align*}
By the Sobolev embedding $H^1(\Omega)\hookrightarrow L^{2d/(d-2)}(\Omega)$, $(c_\delta^{3/2})$ is bounded in $L^2(0,T;L^{2d/(d-2)}(\Omega))$ and thus, by interpolation,
\begin{align*}
  \|c_\delta\|_{L^4(0,T;L^{8d/(3d-4)}(\Omega))}
  \leq \|c_\delta\|_{L^3(0,T;L^{3d/(d-2)}(\Omega))}^{3/4}
  \|c_\delta\|_{L^\infty(0,T;L^2(\Omega))}^{1/4}\leq C,
\end{align*}
achieving the second claim in \eqref{24}.
\end{proof}

We also need estimates for the time derivatives.

\begin{lemma}[Estimates for the time derivatives]
There exists a constant $C>0$ independent of $\delta$ such that
\begin{align}
  \|\partial_t\phi_\delta\|_{L^2(0,T;H^{1}(\Omega)')}
  + \|\partial_t c_\delta\|_{L^{4/3}(0,T;W^{1,8d/(d+4)}(\Omega)')}
  \leq C. \label{a23}
\end{align}
\end{lemma}

\begin{proof}
We only present the proof for $\pa_t c_\delta$, as the proof for $\pa_t\phi_\delta$ is similar. Thanks to \eqref{16}, \eqref{18}, and \eqref{24}, we have for any $\psi\in L^4(0,T;W^{1,8d/(d+4)}(\Omega))$,
\begin{align*}
  \Big|\int_0^T&\int_\Omega\partial_t c_\delta \psi\ dxdt \bigg|
  \leq \|c_\delta\|_{L^4(0,T;L^{8d/(3d-4)}(\Omega))}
  \big\|\sqrt{M_\delta(\phi_\delta)}\big\|_{L^\infty(\Omega_T)}
  \\
  &\phantom{xx}\times\|\nabla\psi\|_{L^4(0,T;L^{8d/(d+4)}(\Omega))}
  \big\|\sqrt{M_\delta(\phi_\delta)}
  \big|\nabla\mu_\delta-c_\delta\nabla \partial_cf_\delta
  (\phi_\delta,c_\delta)\big|\big\|_{L^2(\Omega_T)} \\
  &+ \big\|e^{-[\phi_\delta]_+^1}
  \big\|_{L^\infty(\Omega_T)}
  \|\sqrt{c_\delta}\|_{L^4(0,T;L^{8d/(3d-4)}(\Omega))} \\
  &\phantom{xx}\times\|\sqrt{c_\delta}\nabla\partial_c
  f_\delta(\phi_\delta,c_\delta)\|_{L^2(\Omega_T)}
  \|\nabla\psi\|_{L^4(0,T;L^{8d/(d+4)}(\Omega))} \\
  &+ \delta\|\nabla c_\delta\|_{L^2(\Omega_T)}
  \|\nabla\psi\|_{L^2(\Omega_T)} \leq C.
\end{align*}
This completes the proof of \eqref{a23}.
\end{proof}


\subsection{Compactness argument}

The estimates obtained in the previous subsection allow us to complete the proof of Theorem \ref{main}.

\subsubsection*{Step 1: Convergence of $(\phi_\delta)$}
We deduce from bound \eqref{a23} that there exists a subsequence (not relabeled) such that, as $\delta\to 0$,
\begin{align*}
  \partial_t\phi_\delta\rightharpoonup \partial_t\phi
  &\quad\mbox{weakly in } L^2(0,T;H^{1}(\Omega)'). 
\end{align*}
Taking into account estimate \eqref{a20}, we have, up to a subsequence,
\begin{align}\label{a26}
  \Delta\phi_\delta \rightharpoonup \Delta\phi
  \quad\mbox{weakly in } L^2(\Omega_T).
\end{align}
Estimates \eqref{16}, \eqref{a20}, \eqref{a23}, and the Aubin--Lions lemma show that
\begin{align}\label{27}
  \phi_\delta\rightarrow \phi \quad\mbox{strongly in } L^2(0,T;H^1(\Omega))\cap C([0,T];L^2(\Omega)),
\end{align}
which implies that $\phi_\delta\rightarrow \phi$ a.e.\ in $\Omega_T$. Since $\exp(-[\phi_\delta]_+^1)$ is bounded in $L^\infty(\Omega)$, we have the convergence
\begin{align*}
  e^{-[\phi_\delta]_+^1}\rightarrow e^{-[\phi]_+^1}
  \quad\mbox{strongly in } L^p(\Omega_T),\quad p\in[1,\infty).
\end{align*}
We need to verify that $\phi$ lies between zero and one. Although the proof of the following lemma is very similar to \cite{ElGa96}, we present the full proof for the convenience of the reader.

\begin{lemma}[Upper and lower bounds for $\phi$]\label{lemma4}
The limit function $\phi$ satisfies $0<\phi<1$ a.e.\ in $\Omega_T$.
\end{lemma}

\begin{proof}
For any $\alpha>0$, define the sets
$V_{\alpha,\delta}:=\{(t,x)\in \Omega_T:\phi_\delta(t,x)\geq 1+\alpha \}$. We integrate
$$
  \Phi''_\delta(\phi_\delta(t,x))
  = \frac{1}{M(1-\delta)} = \frac{1}{\delta^2(1-\delta)^2}
  \quad\mbox{for }(t,x)\in V_{\alpha,\delta}
$$
twice to obtain
\begin{align*}
  \Phi_\delta(\phi_\delta(t,x))
  = \int_{1/2}^{\phi_\delta}\int_{1/2}^s
  \frac{drds}{\delta^2(1-\delta)^2}
  = \frac{(\phi_\delta-1/2)^2}{2\delta^2(1-\delta)^2}.
\end{align*}
We infer from the definition of $V_{\alpha,\delta}$ and from the entropy inequality \eqref{19} that
\begin{align*}
  \frac{\alpha^2 {\rm{meas}}(V_{\alpha,\delta})}{2\delta^2(1-\delta)^2}
  \leq \int_{V_{\alpha,\delta}}
  \frac{(\phi_\delta-1/2)^2}{2\delta^2(1-\delta)^2}dxdt
  = \int_{V_{\alpha,\delta}}\Phi_{\delta}(\phi_\delta)(t,x)dxdt
  \leq C.
\end{align*}
In view of the a.e.\ convergence of $(\phi_\delta)$, this yields
\begin{align*}
  {\rm meas}\{(t,x):\phi(t,x)\geq 1+\alpha \}
  = \lim_{\delta\rightarrow 0} {\rm meas}(V_{\alpha,\delta})
  \leq \lim_{\delta\rightarrow 0}
  \frac{2C\delta^2(1-\delta)^2}{\alpha^2} = 0,
\end{align*}
implying that $\phi(t,x)\leq 1+\alpha$ a.e.\ in $\Omega_T$ for all $\alpha>0$. Since $\alpha>0$ is arbitrary,
$\phi(t,x)\le 1$ a.e.\ in $\Omega_T$.
In a similar way, we show that $\phi(t,x)\ge 0$ a.e.\ in $\Omega_T$.

Finally, because of $\Phi_\delta(\phi_\delta)\ge 0$, we can apply Fatou's lemma to conclude that
\begin{align*}
  \int_\Omega \lim_{\delta\rightarrow 0}
  \Phi_\delta(\phi_\delta(t,x))dx
  \leq \lim_{\delta\rightarrow 0}\int_\Omega \Phi_\delta(\phi_\delta(t,x))dx \leq C.
\end{align*}
It follows from $\Phi(0)=\Phi(1)=\infty$ that
\begin{align*}
  \lim_{\delta\rightarrow 0}\Phi_\delta(\phi_\delta)
  = \begin{cases}
  \Phi(\phi) & \mbox{if }0<\phi<1, \\
  \infty &\mbox{if } \phi=0\mbox{ or }\phi=1.
  \end{cases}
\end{align*}
Consequently, ${\rm meas} \{x:\phi(t,x)=0$ or $\phi(t,x)=1\}=0$ for a.e.\ $t\in(0,T)$, concluding the proof.
\end{proof}

\subsubsection*{Step 2: Convergence of $(c_\delta)$}
Bound \eqref{a23} for the time derivative of $c_\delta$ gives the existence of a subsequence (not relabeled) such that, as $\delta\to 0$,
$$
  \partial_t c_\delta\rightharpoonup \partial_t c
  \quad\mbox{weakly in }L^{4/3}(0,T;W^{1,8d/(d+4)}(\Omega)').
$$
To conclude the strong convergence of $(c_\delta)$, we use the ``nonlinear'' Aubin--Lions compactness lemma \cite[Theorem 3]{CJL14}, which provides, in view of the gradient bound in \eqref{24}, a subsequence such that
\begin{align}
  c_\delta\rightarrow c\quad\mbox{strongly in }L^3(\Omega_T). \label{28}
\end{align}
This implies that, again up to a subsequence,
$c_\delta^{3/2}\rightarrow c^{3/2}$ a.e.\ in $\Omega_T$ and, because of \eqref{24},
\begin{align*}
  c_\delta^{3/2}\rightarrow c^{3/2}\quad\mbox{strongly in } L^p(0,T;L^q(\Omega)),\ \mbox{where } p\in [1,2),\ q\in \Big[1,\frac{2d}{d-2}\Big).
\end{align*}
This convergence is sufficient to pass to the limit $\delta\to 0$ in the part of $c_\delta\na\pa_c f(\phi_\delta,c_\delta)$ that contains $c_\delta^{3/2}$. Indeed, we have for any test function $\psi\in C_0^\infty(\Omega_T;\R^d)$, by integrating by parts,
\begin{align*}
  \int_0^T\int_\Omega \sqrt{c_\delta}\nabla c_\delta\cdot \psi dxdt
  &= -\frac{2}{3}\int_0^T\int_\Omega c_\delta^{3/2}\diver\psi dxdt \\
  &\rightarrow -\frac{2}{3}\int_0^T\int_\Omega c^{3/2}\diver\psi dxdt
 = \int_0^T\int_\Omega \sqrt{c}\nabla c\cdot \psi dxdt.
\end{align*}
Actually, since $\na c_\delta^{3/2}=(3/2)\sqrt{c_\delta}\na c_\delta$ is uniformly bounded in $L^2(\Omega_T)$, this convergence holds true even in $L^2(\Omega_T)$:
\begin{align}
  \sqrt{c_\delta}\nabla c_\delta
  \rightharpoonup \sqrt{c}\nabla c \quad\mbox{weakly in }
  L^2(\Omega_T). \label{29}
\end{align}

\subsubsection*{Step 3: Convergence of $M_\delta(\phi_\delta)$}
We deduce from the mean value theorem that, for any $z\in[0,1]$,
\begin{align*}
  |M_\delta(z)-M(z)|&\leq \sup_{0<z<\delta}|M(\delta)-M(z)|
  + \sup_{1-\delta<z<1}|M(1-\delta)-M(z)| \\
  &\leq \sup_{0<z<\delta}M'(\delta_1(z))\delta
  + \sup_{1-\delta<z<1}M'(\delta_2(z))\delta
  \rightarrow 0,
\end{align*}
where $\delta_1(z)\in (z,\delta)$ and $\delta_2(z)\in (1-\delta, z)$. Hence, $M_\delta\rightarrow M$ uniformly in $[0,1]$.
It follows from the continuity of $M$ that
\begin{align}\label{32}
  M_\delta(\phi_\delta)\rightarrow M(\phi)\quad\mbox{a.e.\ in }\Omega_T.
\end{align}
As $M_\delta(\phi_\delta)$ is uniformly bounded in $L^\infty(\Omega_T)$, this implies that
\begin{align}
  M_\delta(\phi_\delta)\rightarrow M(\phi)\quad\mbox{strongly in } L^p(\Omega_T), \quad p\in [1,\infty). \label{31}
\end{align}
Moreover, in view of the bounds $0<\phi<1$ from Lemma \ref{lemma4}, we have
\begin{align*}
  M'_\delta(\phi_\delta)\rightarrow M'(\phi)
  \quad\mbox{a.e.\ in }\Omega_T.
\end{align*}
This yields, together with \eqref{32}, that
\begin{align*}
  \frac{M'_\delta(\phi_\delta)}{\sqrt{M_\delta(\phi_\delta)}}
  \rightarrow \frac{M'(\phi)}{\sqrt{M(\phi)}}\quad\mbox{a.e.\ in } \Omega_T.
\end{align*}
In view of the definition of $M_\delta(\phi_\delta)$,
the singularity of $1/\sqrt{M_\delta(\phi_\delta)}$ is canceled by $M'_\delta(\phi_\delta)$, which provides a uniform $L^\infty(\Omega)$ bound for $M'_\delta(\phi_\delta)/\sqrt{M_\delta(\phi_\delta)}$.
It follows from \eqref{27} and dominated convergence that
\begin{align}
  \frac{M'_\delta(\phi_\delta)}{\sqrt{M_\delta(\phi_\delta)}}
  \nabla\phi_\delta\rightarrow
  \frac{M'(\phi)}{\sqrt{M(\phi)}}\nabla\phi
  \quad\mbox{strongly in } L^2(\Omega_T). \label{35}
\end{align}

According to \eqref{16}, there exists $I\in L^2(\Omega_T)$ such that
\begin{align*}
  \sqrt{M_\delta(\phi_\delta)}\big(\nabla \mu_\delta-c_\delta\nabla\partial_cf_\delta(\phi_\delta,c_\delta) \big)\rightharpoonup I\quad\mbox{weakly in } L^2(\Omega_T). 
\end{align*}
The final step of the proof is concerned with the identification of the limit $I$.

\subsubsection*{Step 4: Identification of $I$}

Let $\Psi\in L^2(0,T;H^{1}(\Omega;\R^d))\cap L^\infty(\Omega_T;\R^d)$
with $\Psi\cdot\nu=0$ on $\partial\Omega$ be given. We compute
\begin{align}\label{33}
  \int_0^T&\int_\Omega \sqrt{M_\delta (\phi_\delta)}
  \big(\nabla\mu_\delta - c_\delta\nabla\partial_cf_\delta
  (\phi_\delta,c_\delta) \big)\cdot\Psi dxdt \\
  &= \int_0^T\int_\Omega \Delta\phi_\delta\diver\big(
  \sqrt{M_\delta (\phi_\delta)}\Psi\big)dxdt
  + \int_0^T\int_\Omega \sqrt{M_\delta (\phi_\delta)}\nabla
  \partial_\phi f_\delta(\phi_\delta,c_\delta)\cdot\Psi dxdt \nonumber \\
  &\phantom{xx}- \int_0^T\int_\Omega
  \sqrt{M_\delta (\phi_\delta)}c_\delta (\nabla c_\delta
  - \nabla\phi_\delta)\cdot\Psi dxdt =: I_1+I_2+I_3. \nonumber
\end{align}

The term $I_1$ can be divided into two parts
\begin{align*}
  I_1 &= \int_0^T\int_\Omega \frac{M'_\delta(\phi_\delta)}{
  2\sqrt{M_\delta(\phi_\delta)}}\nabla\phi_\delta
  \Delta\phi_\delta\cdot\Psi dxdt
  + \int_0^T\int_\Omega \sqrt{M_\delta (\phi_\delta)}
  \Delta\phi_\delta\diver\Psi dxdt \\
  &=: I_{11}+I_{12}.
\end{align*}
It follows from convergences \eqref{a26}, \eqref{27}, and \eqref{35} that
\begin{align*}
  I_{11}\rightarrow \int_0^T\int_\Omega \frac{M'(\phi)}{2\sqrt{M(\phi)}}\nabla\phi\Delta\phi\cdot\Psi dxdt,
\end{align*}
and convergences \eqref{a26} and \eqref{31} yield
\begin{align*}
  I_{12}\rightarrow \int_0^T\int_\Omega \sqrt{M(\phi)}\Delta\phi {\rm{div}} \Psi dxdt.
\end{align*}
Summarizing, this shows that
\begin{align*}
  I_1\rightarrow \int_0^T\int_\Omega \Delta\phi\diver\big(\sqrt{M (\phi)}\Psi\big)dxdt.
\end{align*}

We turn to the term $I_2$. By the definition of $f_\delta$, we have $\nabla\partial_\phi f_{\delta}(\phi_\delta,c_\delta) = \nabla\phi_\delta/\sqrt{M_\delta(\phi_\delta)}
- \nabla c_\delta-2\nabla\phi_\delta$,
which gives
\begin{align*}
  I_2 &= \int_0^T\int_\Omega \nabla\phi_\delta\cdot\Psi dxdt
  + \int_0^T\int_\Omega \frac{M'_\delta(\phi_\delta)}{
  2\sqrt{M_\delta(\phi_\delta)}}\nabla\phi_\delta c_\delta
  \cdot\Psi dxdt \\
  &\phantom{xx}+ \int_0^T\int_\Omega \sqrt{M_\delta (\phi_\delta)}c_\delta\diver
  \Psi  dxdt - 2\int_0^T\int_\Omega \sqrt{M_\delta(\phi_\delta)}
  \nabla \phi_\delta\cdot\Psi dxdt \\
  &=: I_{21}+I_{22}+I_{23}+I_{24}.
\end{align*}
By the convergences \eqref{27}, \eqref{28}, \eqref{31}, and \eqref{35},
\begin{align*}
  I_{21} &\rightarrow \int_0^T\int_\Omega \nabla\phi\cdot\Psi dxdt,
  && I_{22}\rightarrow \int_0^T\int_\Omega \frac{M'(\phi)}{2\sqrt{M(\phi)}}\nabla\phi c\cdot\Psi dxdt, \\
  I_{23} &\rightarrow \int_0^T\int_\Omega \sqrt{M(\phi)}c \diver\Psi dxdt,
  && I_{24}\rightarrow -2\int_0^T\int_\Omega \sqrt{M(\phi)}
  \nabla \phi\cdot\Psi dxdt.
\end{align*}
Therefore,
\begin{align*}
  I_2\rightarrow \int_0^T\int_\Omega \big(\nabla\phi-2\sqrt{M(\phi)}\nabla\phi\big)\cdot\Psi dxdt
  + \int_0^T\int_\Omega c\diver\big(\sqrt{M(\phi)}\Psi\big) dxdt.
\end{align*}

It follows from convergences \eqref{27}, \eqref{28}, \eqref{29}, and \eqref{31} that
\begin{align*}
  I_3&=-\int_0^T\int_\Omega \big(\sqrt{M_\delta(\phi_\delta)}\sqrt{c_\delta}
  (\sqrt{c_\delta} \nabla c_\delta)
  - \sqrt{M_\delta(\phi_\delta)} c_\delta\nabla\phi_\delta\big)
  \cdot\Psi dxdt \\
  &\rightarrow -\int_0^T\int_\Omega \sqrt{M(\phi)}c
  (\nabla c-\nabla\phi)\cdot\Psi dxdt \\
  &=-\int_0^T\int_\Omega \sqrt{M(\phi)}c\nabla\partial_c f(\phi,c)
  \cdot \Psi dxdt.
\end{align*}
Inserting the previous convergence results for $I_1$, $I_2$, and $I_3$ into \eqref{33}, we conclude that
\begin{align*}
  \int_0^T\int_\Omega I\cdot\Psi dxdt
  &= \int_0^T\int_\Omega \big[\Delta\phi{\rm{div}}\big(\sqrt{M(\phi)}\Psi\big)
  + \big(\nabla\phi-2\sqrt{M(\phi)}\nabla\phi\big)\cdot\Psi \\
  &\phantom{xx} + c\diver\big(\sqrt{M(\phi)}\Psi\big)
  - \sqrt{M(\phi)}c\nabla\partial_cf(\phi,c)\cdot\Psi \big]dxdt,
\end{align*}
finishing the proof of Theorem \ref{main}.


\end{document}